\documentclass[twoside, 12pt]{amsart}

\usepackage{amsfonts}
\usepackage{amssymb}
\usepackage{latexsym}
\usepackage{graphics}
\usepackage[2emode]{psfrag}
\usepackage{mathrsfs}
\usepackage{amsthm}
\usepackage{amsmath}
\usepackage[all]{xy}
\usepackage{enumerate}
\usepackage{hyperref}

\begin{document}

\newcommand{\thlabel}[1]{\label{th:#1}}
\newcommand{\thref}[1]{Theorem~\ref{th:#1}}
\newcommand{\selabel}[1]{\label{se:#1}}
\newcommand{\seref}[1]{Section~\ref{se:#1}}
\newcommand{\lelabel}[1]{\label{le:#1}}
\newcommand{\leref}[1]{Lemma~\ref{le:#1}}
\newcommand{\prlabel}[1]{\label{pr:#1}}
\newcommand{\prref}[1]{Proposition~\ref{pr:#1}}
\newcommand{\colabel}[1]{\label{co:#1}}
\newcommand{\coref}[1]{Corollary~\ref{co:#1}}
\newcommand{\relabel}[1]{\label{re:#1}}
\newcommand{\reref}[1]{Remark~\ref{re:#1}}
\newcommand{\exlabel}[1]{\label{ex:#1}}
\newcommand{\exref}[1]{Example~\ref{ex:#1}}
\newcommand{\delabel}[1]{\label{de:#1}}
\newcommand{\deref}[1]{Definition~\ref{de:#1}}
\newcommand{\eqlabel}[1]{\label{eq:#1}}
\newcommand{\equref}[1]{\eqref{eq:#1}}

\newcommand{\Hom}{{\rm Hom}}
\newcommand{\End}{{\rm End}}
\newcommand{\Ext}{{\sf Ext}}
\newcommand{\Fun}{{\rm Fun}}
\newcommand{\Lax}{{\rm Lax}}
\newcommand{\Cat}{{\bf Cat}}
\newcommand{\Mor}{{\rm Mor}\,}
\newcommand{\Aut}{{\rm Aut}\,}
\newcommand{\Hopf}{{\rm Hopf}\,}
\newcommand{\Ann}{{\rm Ann}\,}
\newcommand{\Ker}{{\rm Ker}\,}
\newcommand{\Coker}{{\rm Coker}\,}
\newcommand{\im}{{\rm Im}\,}
\newcommand{\coim}{{\rm Coim}\,}
\newcommand{\Trace}{{\rm Trace}\,}
\newcommand{\Char}{{\rm Char}\,}
\newcommand{\Mod}{{\bf mod}}
\newcommand{\Spec}{{\rm Spec}\,}
\newcommand{\Span}{{\rm Span}\,}
\newcommand{\sgn}{{\rm sgn}\,}
\newcommand{\Id}{{\rm Id}\,}
\newcommand{\Com}{{\rm Com}\,}
\newcommand{\codim}{{\rm codim}}
\newcommand{\Mat}{{\rm Mat}}
\newcommand{\Coint}{{\rm Coint}}
\newcommand{\Incoint}{{\rm Incoint}}
\newcommand{\can}{{\rm can}}
\newcommand{\sign}{{\rm sign}}
\newcommand{\kar}{{\rm kar}}
\newcommand{\rad}{{\rm rad}}
\newcommand{\Rep}{{\rm Rep}}
\newcommand{\Rmod}{{\sf Rmod}}
\newcommand{\Lmod}{{\sf Lmod}}
\newcommand{\Rcom}{{\sf Rcom}}
\newcommand{\Lcom}{{\sf Lcom}}
\newcommand{\Bicom}{{\sf Bicom}}
\newcommand{\RCOM}{{\sf RCOM}}
\newcommand{\LCOM}{{\sf LCOM}}
\newcommand{\Bim}{{\sf Bim}}
\newcommand{\Bic}{{\sf Bic}}
\newcommand{\Frm}{{\sf Frm}}
\newcommand{\Alg}{{\sf Alg}}
\newcommand{\EM}{{\sf EM}}
\newcommand{\REM}{{\sf REM}}
\newcommand{\LEM}{{\sf LEM}}
\newcommand{\fREM}{{\sf fREM}}
\newcommand{\fLEM}{{\sf fLEM}}
\newcommand{\CAT}{{\sf CAT}}
\newcommand{\GAL}{{\sf GAL}}
\newcommand{\Gal}{{\sf Gal}}
\newcommand{\gAL}{{\sf gAL}}
\newcommand{\Tur}{{\sf Tur}}
\newcommand{\Rng}{{\sf Rng}}
\newcommand{\ndi}{{\sf NdI}}
\newcommand{\MFrm}{\MM({\sf Frm}_k)}

\def\Ab{\underline{\underline{\rm Ab}}}
\def\lan{\langle}
\def\ran{\rangle}
\def\ot{\otimes}
\def\uot{\ul{\otimes}}
\def\bul{\bullet}
\def\ubul{\ul{\bullet}}

\def\id{\textrm{{\small 1}\normalsize\!\!1}}
\def\To{{\multimap\!\to}}
\def\bigperp{{\LARGE\textrm{$\perp$}}} 
\newcommand{\QED}{\hspace{\stretch{1}}
\makebox[0mm][r]{$\Box$}\\}
\def\equal#1{\smash{\mathop{=}\limits^{#1}}}

\def\AA{{\mathbb A}}
\def\BB{{\mathbb B}}
\def\CC{{\mathbb C}}
\def\DD{{\mathbb D}}
\def\EE{{\mathbb E}}
\def\FF{{\mathbb F}}
\def\GG{{\mathbb G}}
\def\HH{{\mathbb H}}
\def\II{{\mathbb I}}
\def\JJ{{\mathbb J}}
\def\KK{{\mathbb K}}
\def\LL{{\mathbb L}}
\def\MM{{\mathbb M}}
\def\NN{{\mathbb N}}
\def\OO{{\mathbb O}}
\def\PP{{\mathbb P}}
\def\QQ{{\mathbb Q}}
\def\RR{{\mathbb R}}
\def\SS{{\mathbb S}}
\def\TT{{\mathbb T}}
\def\UU{{\mathbb U}}
\def\VV{{\mathbb V}}
\def\WW{{\mathbb W}}
\def\XX{{\mathbb X}}
\def\YY{{\mathbb Y}}
\def\ZZ{{\mathbb Z}}

\def\aa{{\mathfrak A}}
\def\bb{{\mathfrak B}}
\def\cc{{\mathfrak C}}
\def\dd{{\mathfrak D}}
\def\ee{{\mathfrak E}}
\def\ff{{\mathfrak F}}
\def\gg{{\mathfrak G}}
\def\hh{{\mathfrak H}}
\def\ii{{\mathfrak I}}
\def\jj{{\mathfrak J}}
\def\kk{{\mathfrak K}}
\def\ll{{\mathfrak L}}
\def\mm{{\mathfrak M}}
\def\nn{{\mathfrak N}}
\def\oo{{\mathfrak O}}
\def\pp{{\mathfrak P}}
\def\qq{{\mathfrak Q}}
\def\rr{{\mathfrak R}}
\def\ss{{\mathfrak S}}
\def\tt{{\mathfrak T}}
\def\uu{{\mathfrak U}}
\def\vv{{\mathfrak V}}
\def\ww{{\mathfrak W}}
\def\xx{{\mathfrak X}}
\def\yy{{\mathfrak Y}}
\def\zz{{\mathfrak Z}}

\def\aaa{{\mathfrak a}}
\def\bbb{{\mathfrak b}}
\def\ccc{{\mathfrak c}}
\def\ddd{{\mathfrak d}}
\def\eee{{\mathfrak e}}
\def\fff{{\mathfrak f}}
\def\ggg{{\mathfrak g}}
\def\hhh{{\mathfrak h}}
\def\iii{{\mathfrak i}}
\def\jjj{{\mathfrak j}}
\def\kkk{{\mathfrak k}}
\def\lll{{\mathfrak l}}
\def\mmm{{\mathfrak m}}
\def\nnn{{\mathfrak n}}
\def\ooo{{\mathfrak o}}
\def\ppp{{\mathfrak p}}
\def\qqq{{\mathfrak q}}
\def\rrr{{\mathfrak r}}
\def\sss{{\mathfrak s}}
\def\ttt{{\mathfrak t}}
\def\uuu{{\mathfrak u}}
\def\vvv{{\mathfrak v}}
\def\www{{\mathfrak w}}
\def\xxx{{\mathfrak x}}
\def\yyy{{\mathfrak y}}
\def\zzz{{\mathfrak z}}

\newcommand{\aA}{\mathscr{A}}
\newcommand{\bB}{\mathscr{B}}
\newcommand{\cC}{\mathscr{C}}
\newcommand{\dD}{\mathscr{D}}
\newcommand{\eE}{\mathscr{E}}
\newcommand{\fF}{\mathscr{F}}
\newcommand{\gG}{\mathscr{G}}
\newcommand{\hH}{\mathscr{H}}
\newcommand{\iI}{\mathscr{I}}
\newcommand{\jJ}{\mathscr{J}}
\newcommand{\kK}{\mathscr{K}}
\newcommand{\lL}{\mathscr{L}}
\newcommand{\mM}{\mathscr{M}}
\newcommand{\nN}{\mathscr{N}}
\newcommand{\oO}{\mathscr{O}}
\newcommand{\pP}{\mathscr{P}}
\newcommand{\qQ}{\mathscr{Q}}
\newcommand{\rR}{\mathscr{R}}
\newcommand{\sS}{\mathscr{S}}
\newcommand{\tT}{\mathscr{T}}
\newcommand{\uU}{\mathscr{U}}
\newcommand{\vV}{\mathscr{V}}
\newcommand{\wW}{\mathscr{W}}
\newcommand{\xX}{\mathscr{X}}
\newcommand{\yY}{\mathscr{Y}}
\newcommand{\zZ}{\mathscr{Z}}

\newcommand{\Aa}{\mathcal{A}}
\newcommand{\Bb}{\mathcal{B}}
\newcommand{\Cc}{\mathcal{C}}
\newcommand{\Dd}{\mathcal{D}}
\newcommand{\Ee}{\mathcal{E}}
\newcommand{\Ff}{\mathcal{F}}
\newcommand{\Gg}{\mathcal{G}}
\newcommand{\Hh}{\mathcal{H}}
\newcommand{\Ii}{\mathcal{I}}
\newcommand{\Jj}{\mathcal{J}}
\newcommand{\Kk}{\mathcal{K}}
\newcommand{\Ll}{\mathcal{L}}
\newcommand{\Mm}{\mathcal{M}}
\newcommand{\Nn}{\mathcal{N}}
\newcommand{\Oo}{\mathcal{O}}
\newcommand{\Pp}{\mathcal{P}}
\newcommand{\Qq}{\mathcal{Q}}
\newcommand{\Rr}{\mathcal{R}}
\newcommand{\Ss}{\mathcal{S}}
\newcommand{\Tt}{\mathcal{T}}
\newcommand{\Uu}{\mathcal{U}}
\newcommand{\Vv}{\mathcal{V}}
\newcommand{\Ww}{\mathcal{W}}
\newcommand{\Xx}{\mathcal{X}}
\newcommand{\Yy}{\mathcal{Y}}
\newcommand{\Zz}{\mathcal{Z}}

\def\mor{\xymatrix{\ar[r]|{|}&}}
\def\morr #1{\xymatrix{\ar[r]|{|}^{#1}&}}

\def\units{{\mathbb G}_m}
\def\ract{\hbox{$\leftharpoonup$}}
\def\lact{\hbox{$\rightharpoonup$}}

\def\ractd{\hbox{$\lhd$}}
\def\lactd{\hbox{$\rhd$}}

\def\*C{{}^*\hspace*{-1pt}{\Cc}}

\def\text#1{{\rm {\rm #1}}}
\def\upsmile#1{\raisebox{0.2ex}{\ensuremath{\stackrel{\smile}{#1}}}}

\def\smashco{\mathrel>\joinrel\mathrel\triangleleft}
\def\cosmash{\mathrel\triangleright\joinrel\mathrel<}

\def\ol{\overline}
\def\ul{\underline}
\def\dul#1{\underline{\underline{#1}}}
\def\Nat{\dul{\rm Nat}}
\def\Set{\dul{\rm Set}}

\renewcommand{\subjclassname}{\textup{2000} Mathematics Subject
     Classification}

\newtheorem{proposition}{Proposition}[section] 
\newtheorem{lemma}[proposition]{Lemma}
\newtheorem{corollary}[proposition]{Corollary}
\newtheorem{theorem}[proposition]{Theorem}

\theoremstyle{definition}
\newtheorem{Definition}[proposition]{Definition}
\newtheorem{example}[proposition]{Example}
\newtheorem{examples}[proposition]{Examples}

\theoremstyle{remark}
\newtheorem{remarks}[proposition]{Remarks}
\newtheorem{remark}[proposition]{Remark}
\newtheorem{notation}[proposition]{Notation}

\title{Multiplier Hopf and bi-algebras}
\date{\today}
\author{K. Janssen}
\author{J. Vercruysse}
\address{Faculty of Engineering, Vrije Universiteit Brussel (VUB), B-1050 Brussels, Belgium}
\email{krjansse@vub.ac.be}
\urladdr{homepages.vub.ac.be/\~{}krjansse}
\email{jvercruy@vub.ac.be}
\urladdr{homepages.vub.ac.be/\~{}jvercruy}


\begin{abstract}
We propose a categorical interpretation of multiplier Hopf algebras, in analogy to usual Hopf algebras and bialgebras. Since the introduction of multiplier Hopf algebras by Van Daele in \cite{VD:Mult} such a categorical interpretation has been missing. We show that a multiplier Hopf algebra can be understood as a coalgebra with antipode in a certain monoidal category of algebras. We show that a (possibly non-unital, idempotent, non-degenerate, $k$-projective) algebra over a commutative ring $k$ is a multiplier bialgebra if and only if the category of its algebra extensions and both the categories of its left and right modules are monoidal and fit, together with the category of $k$-modules, into a diagram of strict monoidal forgetful functors.
\end{abstract}

\maketitle

\section*{Introduction}

A Hopf algebra over a commutative ring $k$ is defined as a $k$-bialgebra, equipped with an antipode map. A $k$-bialgebra can be understood as a coalgebra (or comonoid) in the monoidal category of $k$-algebras; an antipode is the inverse of the identity map in the convolution algebra of $k$-endomorphisms of the $k$-bialgebra.
From the module theoretic point of view, a $k$-bialgebra can also be understood as a $k$-algebra that turns its category of left (or, equivalently, right) modules into a monoidal category, such that the forgetful functor to the category of $k$-modules is a strict monoidal functor.

During the last decades, many generalizations of and variations on the definition of a Hopf algebra have emerged in the literature, such as quasi-Hopf algebras \cite{Drin:QH}, weak Hopf algebras \cite{bohm:weakhopf}, Hopf algebroids \cite{BohmSzl:hgdax} and  
Hopf group (co)algebras \cite{CDL}. 
For most of these notions, the above categorical and module theoretic interpretations remain valid in a certain form, and in some cases, this was exactly the motivation to introduce such a new Hopf-type algebraic structure.

Multiplier Hopf algebras were introduced by Van Daele in \cite{VD:Mult}, motivated by the theory of (discrete) quantum groups. The initial data of a multiplier Hopf algebra are a non-unital algebra $A$, a so-called comultiplication map $\Delta:A\to \MM(A\ot A)$, where $\MM(A\ot A)$ is the {\em multiplier algebra} of $A\ot A$, which is the ``largest'' unital algebra containing $A\ot A$ as a two-sided ideal, and certain bijective endomorphisms on $A\ot A$. It is well-known that the dual of a Hopf algebra is itself a Hopf algebra only if the original Hopf algebra is finitely generated and projective over its base ring. A very nice feature of the theory of multiplier Hopf algebras is that it lifts this duality to the infinite dimensional case.
In particular, the (reduced) dual of a co-Frobenius Hopf algebra is a multiplier Hopf algebra, rather than a usual Hopf algebra.

From the module theoretic point of view, some aspects in the definition of a multiplier Hopf algebra remain however not completely clear. For example, a multiplier Hopf algebra is introduced without defining first an appropriate notion of a ``multiplier bialgebra''. In particular, 
from the definition of a multiplier Hopf algebra a counit can be constructed, rather than being given as part of the initial data.
Furthermore, a categorical characterization of multiplier Hopf algebras as in the classical and more general cases as mentioned before seems to be missing or at least unclear at the moment. In this paper we try to shed some light on this situation.

Our paper is organized as follows. In the first Section, we recall some notions related to non-unital algebras and non-unital modules. We repeat the construction of the multiplier algebra of a non-degenerate (idempotent) algebra, and show how this notion is related to extensions of non-unital algebras. We show how non-degenerate idempotent $k$-projective algebras constitute a monoidal category. 

In the second Section we then introduce the notion of a {\em multiplier bialgebra} as a coalgebra in this monoidal category of non-degenerate idempotent $k$-projective algebras. We show that a non-degenerate idempotent $k$-projective algebra is a multiplier bialgebra if and only if the category of its extensions, as well as the categories of its left and right modules are monoidal and 
fit, together with the category of modules over the commutative base ring, into a diagram of strict monoidal forgetful functors (see \thref{catint}). The main difference from the unital case is that monoidality of the category of left modules is not equivalent to monoidality of the category of right modules.

In the last Section, we recall the definition of a multiplier Hopf algebra by Van Daele. We show that a multiplier Hopf algebra is always a multiplier bialgebra and we give an interpretation of the antipode as a type of convolution inverse. We conclude the paper by providing a categorical way to introduce the notions of module algebra and comodule algebra over a multiplier bialgebra, which in the multiplier Hopf algebra case were studied in \cite{VDZha:Galois}.

\subsection*{Notation}
Throughout, let $k$ be a commutative ring. 
All modules are over $k$ and linear means $k$-linear.
Unadorned tensor products are supposed to be over $k$. 
$\Mm_k$ denotes the category of $k$-modules. By an \emph{algebra} we mean a $k$-module $A$ equipped with an associative $k$-linear map $\mu_A: A\ot A\to A$; it is not assumed to possess a unit.
An \emph{algebra map} between two algebras is a $k$-linear map that preserves the multiplication.
A \emph{unital} algebra is an algebra $A$ with unit element $1_A$. An algebra map between two unital algebras that maps the one unit element to the other, will be called a \emph{unital} algebra map.
A right $A$-module $M$ is a $k$-module equipped with a $k$-linear map $\mu_{M,A}:M\ot A\to M$ such that the associativity condition
$\mu_{M,A}\circ (\mu_{M,A}\ot A)= \mu_{M,A} \circ(M\ot \mu_A)$
holds.   
The $k$-module of right $A$-linear (resp. left $B$-linear, $(B,A)$-bilinear) maps between two right $A$-modules (resp. left $B$-modules, $(B,A)$-bimodules) $M$ and $N$ is denoted by $\Hom_A(M,N)$ (resp. ${_B\Hom}(M,N)$, 
${_B\Hom_A}(M,N)$). We will shortly denote $\Hom_k(M,N)=\Hom(M,N)$.
Note that for any three $k$-modules $M,N$ and $P$, where $M$ is $k$-projective, the following $k$-linear maps are injective:
\begin{eqnarray}\eqlabel{alphabeta}
&&\hspace{-1cm}
\alpha_{M,N,P}:M\ot \Hom(N,P)\to \Hom(N,M\ot P),~ \alpha_{M,N,P}(m\ot f)(n)=m\ot f(n),\\
&&\hspace{-1cm}
\beta_{N,P,M}: \Hom(N,P)\ot M\to \Hom(N, P\ot M),~ \beta_{N,P,M}(f\ot m)(n)=f(n) \ot m.
\end{eqnarray}  
In fact, this is already the case if $M$ is a locally projective $k$-module.
For an object $X$ in a category, we denote the identity morphism on $X$ also by $X$.

\section{Non-unital algebras and extensions}

\subsection{Non-degenerate idempotent algebras}
Let $A$ be an algebra, we say that $A$ is \emph{idempotent} if $A=A^2:=\{ \sum_i a_ia'_i ~|~ a_i, a'_i\in A\}$.  
We will use the following Sweedler-type notation for idempotent algebras: for an element $a\in A$ we denote by  
$\sum a^1a^2$ a (non-unique) element of $A^2$ such that $\sum a^1a^2=a$.
The algebra $A$ is said to be \emph{non-degenerate} if, for all $a\in A$, we have that $a=0$ if $ab=0$ for all $b\in A$ or $ba=0$ for all $b\in A$.

\begin{example}
Clearly, if $A$ is a unital algebra, then $A$ is a non-degenerate idempotent algebra. More generally, $A$ is called an \emph{algebra with right (resp. left) local units} if, for all $a\in A$, there exists an element $e \in A$ such that $ae=a$ (resp. $ea=a$). If $A$ has right (or left) local units, then $A$ is non-degenerate and idempotent.  A \emph{complete set of right (resp. left) local units for $A$} is a subset $E\subset A$ such that, for every $a\in A$, we can find at least one right (resp. left) local unit $e\in E$.
\end{example}

Let $A$ be an algebra. 
The category of all right $A$-modules and right $A$-linear maps is denoted by $\widetilde{\Mm}_A$. A right $A$-module $M$ is called \emph{idempotent} if $M=MA:=\{\sum_i m_ia_i ~|~ m_i\in M, a_i\in A\}$. A right $A$-module $M$ is said to be \emph{non-degenerate} if for all $m\in M$ the equalities $ma=0$ for all $a\in A$ imply that $m=0$.
The full subcategory of $\widetilde{\Mm}_A$ consisting of all non-degenerate idempotent $k$-projective right $A$-modules, is denoted  by $\Mm_A$. It is obvious that every non-degenerate idempotent $k$-projective algebra $A$ is in $\Mm_A$, taking $\mu_{A,A}=\mu_A$. Similarly, we can introduce the category of non-degenerate idempotent $k$-projective left $A$-modules ${_A\Mm}$ and the category of non-degenerate idempotent $k$-projective $(A,B)$-bimodules ${_A\Mm_B}$, with $B$ another algebra.

\begin{example}\exlabel{locunit}
If $A$ is an algebra with right (resp. left) local units, then the idempotent right (resp. left) $A$-modules are exactly those right (resp. left) $A$-modules $M$ such that for all $m\in M$, there exists an element $e\in A$ such that $me= m$ (resp. $em= m$). We say that \emph{$A$ acts with right (resp. left) local units on $M$}. Remark that an idempotent right (resp. left) module over an algebra with right (resp. left) local units is automatically non-degenerate. The converse is not true: consider an algebra with right local units $A$, and let $\End_A(A)$ be the right $A$-module by putting $(f\cdot a)(b)= f(ab)$ for all $a,b\in A$ and $f\in \End_A(A)$. Then $\End_A(A)$ is non-degenerate as right $A$-module, but $A$ can only act with right local units on $A\in\End_A(A)$ if $A$ has a (global) unit element.\\
If $A$ is a unital algebra, then every idempotent right (resp. left) $A$-module $M$ is also unital, in the sense that $m1_A=m$ (resp. $1_Am=m$), for all $m\in M$.
\end{example}

Now consider two algebras $A$ and $B$. We say that $A$ is an \emph{(algebra) extension} of $B$ (or shortly a \emph{$B$-extension}) if $A$ is a $B$-bimodule and the multiplication of $A$ is $B$-bilinear and $B$-balanced (the latter meaning that $(a\cdot b)a' = a(b\cdot a')$, for all $a, a'\in A$ and $b\in B$).
We call $A$ a {\em non-degenerate} (resp.\ {\em idempotent}, {\em $k$-projective}) extension of $B$ if $A$ is an extension of $B$ such that 
$A$ is non-degenerate (resp.\ idempotent, $k$-projective) as a $B$-bimodule.
Note that an idempotent (resp. non-degenerate, $k$-projective) algebra is in a canonical way an idempotent (resp. non-degenerate, $k$-projective) extension of itself.
It is also obvious that a $k$-projective algebra that is a $B$-extension is a $k$-projective $B$-extension.
\begin{lemma}\lelabel{111}
Let $A$ be a non-degenerate algebra and an idempotent extension of $B$, then $A$ is a non-degenerate idempotent extension of $B$.
\end{lemma}

\begin{proof}
Take any $a\in A$ such that $a\cdot b= 0$ for all $b\in B$. We have to show that $a=0$. 
Take any other $a'\in A$.
Since $A$ is an idempotent left $B$-module we can write $a'=\sum_i b_i\cdot a'_i\in BA$. Hence we find that $a a'= a (\sum_i b_i\cdot a'_i) = \sum_i (a\cdot b_i) a'_i = 0$, where we used in the second equation the $B$-balancedness of the product of $A$. Since $A$ is a non-degenerate algebra we conclude that $a=0$.
\end{proof}

Let now $A$ and $A'$ be two (non-degenerate) idempotent extensions of $B$. A \emph{transformation} from $A$ to $A'$ is a $B$-bilinear algebra map $t:A\to A'$.
The category with as objects non-degenerate idempotent $k$-projective extensions of an algebra $B$ and as morphisms transformations between the extensions is denoted by $B$-$\Ext$. Note that the canonical extension $B$ of $B$ is only an object in $B$-$\Ext$ provided that $B$ is a non-degenerate idempotent $k$-projective algebra.

\begin{example}\exlabel{extension}
Suppose that $A$ and $B$ are unital algebras. Then $A$ is a (non-degenerate) idempotent extension of $B$ if and only if there is a unital algebra map $f:B\to A$. Indeed, if $A$ is an idempotent extension of $B$, then define $f(b)=1_A\cdot b=b\cdot 1_A$.

Let now $A$ and $B$ be algebras with right and left local units. An algebra map $f:B\to A$ is called a \emph{morphism of algebras with right (resp. left) local units} if there exists a complete set of right (resp. left) local units $E\subset B$ for $B$ such that $f(E)\subset A$ is a complete set of right (resp. left) local units for $A$. The algebra map $f$ induces a natural $B$-bimodule structure on $A$ by putting $b\cdot a\cdot b' = f(b)af(b')$ for all $b,b'\in  B$ and $a\in A$. One can easily see that if $f$ is both a morphism of algebras with right local units and a morphism of algebras with left local units, then this bimodule structure is idempotent (in fact, $B$ acts with left and right local units on $A$), hence $A$ is an idempotent extension of $B$.
\end{example}

\begin{remark}
Not every idempotent extension of algebras with right (or left) local units, or, more generally, of non-degenerate idempotent algebras, is induced by an algebra map as in \exref{extension}. However, in \seref{extensions} we will show that they are induced by a more general type of morphism.
\end{remark}

\subsection{Multiplier algebras}

The notion of a multiplier algebra of a (possibly non-unital) algebra goes back to G.\ Hochschild \cite{Ho:Cohom} in his work on cohomology and extensions, and to B.\ E.\ Johnson \cite{Joh:centr} in his study of centralizers in topological algebra. A first pure algebraic investigation of this notion was initiated by J.\ Dauns in \cite{Dau:mult}. In this Section we will recall the construction of a multiplier algebra for sake of completeness and in order to introduce the necessary notation.

Given an algebra $A$, we consider the $k$-modules
$L(A)=\End_A(A), R(A)={_A\End}(A)$ and $H(A)={_A\Hom_A}(A\ot A,A)$. We have natural linear maps
\begin{eqnarray}\eqlabel{lambdarho}
L:A\to L(A),\ L(b)(a)=\lambda_b(a)=ba;\\
R:A\to R(A),\ R(b)(a)=\rho_b(a)=ab.\nonumber
\end{eqnarray}
Now consider the linear maps
\begin{eqnarray}
\ol{(-)}:R(A)\to H(A),\ \bar{\rho}(a\ot b)=\rho(a)b,\ {\rm for}\ \rho\in R(A); \\
\ul{(-)}:L(A)\to H(A),\ \ul{\lambda}(a\ot b)=a\lambda(b),\ {\rm for}\ \lambda\in L(A).
\end{eqnarray}
We define $\MM(A)$, the \emph{multiplier algebra} of $A$, as the pullback of $\ol{(-)}$ and $\ul{(-)}$ in $\Mm_k$, i.e.\ the pullback of the following diagram.
\begin{eqnarray}\eqlabel{pullback}
\xymatrix{
\MM(A) \ar[rr] \ar[d] && R(A) \ar[d]^{\ol{(-)}}\\
L(A) \ar[rr]_-{\ul{(-)}} && H(A)
}
\end{eqnarray}
Remark that if $A$ is unital then $A\cong L(A)\cong R(A)\cong H(A)$ in a canonical way, 
hence also $\MM(A)\cong A$.
We can understand $\MM(A)$ as the set of pairs $(\lambda,\rho)$, where $\lambda\in L(A)$ and $\rho\in R(A)$, such that 
\begin{equation}\eqlabel{compatlambdarho}
a\lambda(b)=\rho(a)b,
\end{equation}
for all $a,b\in A$. Elements of $\MM(A)$ are called \emph{multipliers}; those of $L(A)$ and $R(A)$ are called \emph{left}, resp. \emph{right multipliers}. The following Proposition collects some elementary properties of $\MM(A)$ and its relation to $A$.

\begin{proposition}\prlabel{multialg}
With notation as above, the following statements hold:
\begin{enumerate}[(i)]
\item $\MM(A)$ is a unital algebra, with multiplication
\begin{equation}\eqlabel{xy}
xy=(\lambda_x\circ\lambda_y,\rho_y\circ\rho_x)
\end{equation}
for all $x=(\lambda_x,\rho_x)$ and $y=(\lambda_y,\rho_y)$ in $\MM(A)$, and unity $1:=1_{\MM(A)}=(A,A)$;
\item there are natural algebra maps 
\begin{eqnarray}
\iota_A:A\to \MM(A), && \iota_A(a)=(\lambda_a,\rho_a),\\
\pi^\ell_A : \MM(A)\to L(A), &&\pi^\ell_A(x)=\lambda,\\
\pi^r_A : \MM(A)\to R(A)^{\rm op}, &&\pi^r_A(x)=\rho, 
\end{eqnarray}
where $x=(\lambda,\rho)\in\MM(A),a\in A$ and the multiplication on $L(A)$ and $R(A)$ is given by composition;
\item $A$ is non-degenerate if and only if 
$L=\pi^\ell_A\circ \iota_A$ and $R=\pi^r_A\circ \iota_A$ are injective; in this case $\iota_A$ is injective as well;
\item $\MM(A)$ is an $A$-bimodule with actions given by
\begin{equation} \eqlabel{AbimMA}
x\ract a:=x\iota_A(a)=\iota_A(\lambda(a)),\quad a\lact x=\iota_A(a)x=\iota_A(\rho(a)),
\end{equation}
for all $x=(\lambda,\rho)\in\MM(A)$ and $a\in A$;
\item $A$ is a non-degenerate idempotent $\MM(A)$-extension, where the left and right actions of $x=(\lambda,\rho)\in\MM(A)$ on $a\in A$ are given by
$$x\lactd a = \lambda(a),\qquad a\ractd x=\rho(a);$$
\item if $\iota_A$ is injective (e.g. if $A$ is non-degenerate) then $A$ is a two-sided ideal in $\MM(A)$;
\item For any $M\in\Mm_A$, we have that the formula
$$m\ractd x =\sum_im_i (a_i\ractd x),$$
where $m=\sum_im_ia_i\in MA$ and $x\in \MM(A)$, defines a unital right $\MM(A)$-action on $M$. 
Hence
$\Mm_A$ is a full subcategory of $\Mm_{\MM(A)}$, the category of unital right $\MM(A)$-modules.
Moreover, for all $M\in\Mm_A,m\in M, a\in A$ and $x\in \MM(A)$, we have
\begin{eqnarray}\eqlabel{hulp}
ma&=& m\ractd \iota_A(a),\\
(m\ractd x) a &=& m (x\lactd a),\nonumber
\end{eqnarray}
i.e. the action of $A$ on $M$ is $\MM(A)$-balanced, and
\begin{eqnarray}\eqlabel{HULP}
(ma)\ractd x&=& m(a\ractd x).
\end{eqnarray}
\end{enumerate}
\end{proposition}

\begin{proof}
\ul{(i)}. 
By a double application of \equref{compatlambdarho}, we find that
\begin{eqnarray*}
a(\lambda_x\circ\lambda_y)(b)&=&\rho_x(a)\lambda_y(b)=(\rho_y\circ\rho_x)(a)b,
\end{eqnarray*}
hence $xy$ defined by formula \equref{xy} is indeed a multiplier.\\
\ul{(iii)}. This follows from the definition of the maps $L$ and $R$, see \equref{lambdarho}.\\
\ul{(iv)}. This bimodule action is induced by the algebra map $\iota_A$. Explicitly, from the right $A$-linearity of $\lambda$ and the left $A$-linearity of $\rho$, we obtain that $\ract$ and $\lact$ are indeed actions. Let us verify that this defines indeed an $A$-bimodule structure:
$$ a\lact (x \ract b) = a\lact (\iota_A (\lambda(b)))=\iota_A(a\lambda(b))=\iota_A(\rho(a)b)=(\iota_A(\rho(a))) \ract b=(a\lact x)\ract b,$$
where we used \equref{AbimMA} in the third equality.\\
\ul{(ii), (v), and (vi)} are obvious.\\
\ul{(vii)}. We have to check that the action of $\MM(A)$ on $M\in\Mm_A$ is well-defined. Suppose that $m=\sum_im_ia_i=0$ and take $x\in \MM(A)$. For all $a\in A$, we then have $(\sum_im_i(a_i\ractd x)) a=\sum_i m_i ((a_i\ractd x) a)=\sum_i m_i(a_i (x\lactd a))=m(x\lactd a)=0$, hence $\sum_im_i(a_i\ractd x)=0$ by the non-degeneracy of $M$ as right $A$-module. To check the balancedness, we compute
$$(m\ractd x) a=(\sum_im_i (a_i \ractd x)) a = \sum_i m_i (a_i (x\lactd a))= m (x\lactd a).$$
The remaining assertions follow immediately.
\end{proof}

\subsection{Non-degenerate idempotent extensions}\selabel{extensions}

\begin{lemma}\lelabel{extvsmod}
Let $A$ and $B$ be algebras.
\begin{enumerate}[(i)]
\item There is a bijective correspondence between algebra maps $\ell:B\to L(A)$ 
and left $B$-module structures on $A$ such that $\mu_A$ is left $B$-linear (i.e.\ $\mu_A(ba\ot a')= b\mu_A(a\ot a')$, for all $a,a'\in A$ and $b\in B$);
\item there  is a bijective correspondence between algebra maps $r:B\to R(A)^{\rm op}$ 
and right $B$-module structures on $A$ such that
$\mu_A$ is right $B$-linear (i.e.\ $\mu_A(a\ot a'b)=\mu_A(a\ot a')b$, for all $a,a'\in A$ and $b\in B$);
\item there  is a bijective correspondence between algebra maps $f:B\to \MM(A)$ and $B$-extension structures on $A$.
\end{enumerate}
\end{lemma}

\begin{proof}
$\ul{(i)}$. Suppose that the map $\ell$ exists, then we define for all $a\in A$ and $b\in B$, 
$$b\cdot a:= \ell(b)(a).$$
If we now take another element $b'\in B$, then we find 
$$b\cdot (b'\cdot a)= \ell(b)(\ell(b')(a))= \ell(bb')(a) = (bb')\cdot a,$$
which shows that the action is associative. Since for any $b\in B$, $\ell(b)\in L(A)$ is right $A$-linear, we obtain that $\mu_A$ is left $B$-linear.
Conversely, if $A$ is a left $B$-module, then define
$$\ell(b)(a):= b\cdot a,$$
for all $b\in B$ and $a\in A$. Since $\mu_A$ is left $B$-linear this map is well-defined. Similar computations as above show that the associativity of the action implies that $\ell$ is an algebra map from $B$ to $L(A)$.\\
$\ul{(ii)}$. This follows symmetrically.\\
$\ul{(iii)}$. Suppose that the map $f$ exists. Then we obtain two algebra maps $\pi^\ell_A\circ f:B\to L(A)$ and $\pi^r_A\circ f:B\to R(A)^{\rm op}$. Hence by the first two parts, $A$ will be a left and right $B$-module with $B$-actions given by $b\cdot a= \pi^\ell_A(f(b))(a)=f(b)\lactd a$ and $a\cdot b=\pi^r_A(f(b))(a)=a\ractd f(b)$ for all $a\in A$ and $b\in B$, where $\lactd$ and $\ractd$ are the left and right $\MM(A)$-actions on $A$ (see \prref{multialg} (v)).
Since $A$ is an $\MM(A)$-extension it follows immediately that $A$ is a $B$-extension.
Conversely, if $A$ is an extension of $B$, then $A$ is in particular a $B$-bimodule and $\mu_A$ is $B$-bilinear, so by the first two parts, we have algebra maps $\ell:B\to L(A)$ and $r:B\to R(A)^{\rm op}$. Now define for all $b\in B$, $f(b):=(\ell(b),r(b))$. Then the $B$-balancedness of the multiplication of $A$ implies that the image of $f$ lies in $\MM(A)$. $f$ is an algebra map since $\ell$ and $r$ are algebra maps.
\end{proof}

\begin{remark}\relabel{nondegenerateVD}
In \cite[Appendix]{VD:Mult}, an algebra map $f:B\to \MM(A)$ such that $f(B)\lactd A=A=A\ractd f(B)$ was called \emph{non-degenerate}. This coincides with our notion of $A$ being an idempotent $B$-extension, which is by \leref{111} the same as a non-degenerate idempotent $B$-extension if $A$ is a non-degenerate algebra. The latter is always asumed in \cite{VD:Mult}.
\end{remark}

\begin{lemma}\lelabel{extvsfunct}
Let $B$ be an algebra and $A$ a non-degenerate idempotent $k$-projective algebra.
\begin{enumerate}[(i)]
\item If there exists an algebra map $f:B\to \MM(A)$ such that $A$ is a non-degenerate idempotent ($k$-projective) right $B$-module with action induced by $f$, then there is a functor $F_r:\Mm_A \to \Mm_B$ such that the following diagram commutes.
\[
\xymatrix{
\Mm_A \ar[rr]^-{F_r} \ar[dr] && \Mm_B \ar[dl] \\
&\Mm_k
}
 \]
Here the unlabeled arrows are forgetful functors. 
\item If there exists a functor $F_r:\Mm_A\to\Mm_B$ rendering commutative the above diagram, then there is an algebra map $r:B\to R(A)^{op}$ such that $A$ is a non-degenerate idempotent $k$-projective right $B$-module with action induced by $r$.
\end{enumerate}
 \end{lemma}

\begin{proof}
$\ul{(i)}$. Recall from \leref{extvsmod} (iii) that $A$ is a $B$-bimodule, with $B$-actions
given by $b\cdot a=f(b)\lactd a$ and $a\cdot b=a\ractd f(b)$ for all $a\in A$ and $b\in B$,
and $\mu_A$ is $B$-balanced.
For any $M\in\Mm_A$, consider $m\in M$ and $b\in B$ and define $m\cdot b:= m\ractd f(b)$, where we use the action of $\MM(A)$ on $M$ defined in \prref{multialg} (vii).
Since the action of $A$ on $M$ is $\MM(A)$-balanced, it will be $B$-balanced as well, i.e. for all $m\in M,a\in A$ and $b\in B$
\begin{eqnarray}\eqlabel{huulp}
(m\cdot b)a= m(b\cdot a).
\end{eqnarray}
Let us verify that this action of $B$ on $M$ is non-degenerate and idempotent. Using the idempotency of $M$ as right $A$-module and of $A$ as right $B$-module, we find that $M= MA = M(AB) = (MA)B =MB$, where the third equality, meaning $(ma)\cdot b= m(a\cdot b)$ for all $m\in M,a\in A$ and $b\in B$, follows directly from the definition of the right $B$-action on $M$ and \equref{hulp}.
Indeed, we have that $(ma)\cdot b=(ma)\ractd f(b)= m(a\ractd f(b))=m(a\cdot b)$.
To prove the non-degeneracy of the right $B$-action on $M$, take any $m\in M$ such that $m\cdot b=0$ for all $b\in B$. Since $A$ is idempotent as left $B$-module, we can write any $a\in A$ as $a=\sum_i b_i \cdot a_i$, and obtain
$$ma=\sum_i m (b_i\cdot a_i) \equal{\equref{huulp}} \sum_i (m\cdot b_i) a_i =0.$$
Since this holds for all $a\in A$, we obtain by the non-degeneracy of $M$ as right $A$-module that $m=0$.\\
$\ul{(ii)}$.
Since $A$ is a non-degenerate idempotent $k$-projective algebra, we have in particular that $A\in\Mm_A$. Hence $F_r(A)=A$ is a non-degenerate idempotent ($k$-projective) right $B$-module. By \leref{extvsmod}, we know that any right $B$-module structure on $A$ is induced by a map $r:B\to R(A)^{\rm op}$, provided that the multiplication map $\mu_A$ of $A$ is right $B$-linear. That the latter is the case follows from the functoriality of $F_r$. Indeed, fix $a\in A$ and consider the right $A$-linear map $\lambda_a\in L(A)$, defined by $\lambda_a(a')=aa'$, for all $a'\in A$. We then have that $F_r(\lambda_a)=\lambda_a$ is a morphism in $\Mm_B$, i.e. $\lambda_a(a'\cdot b)=\lambda_a(a')\cdot b$, for all $a'\in A$. Thus, for all $a,a'\in A$ we have that $a(a'\cdot b)=(aa')\cdot b$, that is, $\mu_A$ is right $B$-linear.
\end{proof}
\begin{theorem}\thlabel{charext}
Let $A$ be a non-degenerate idempotent $k$-projective algebra and $B$ any algebra.
Then there is a bijective correspondence between the following sets of data:
\begin{enumerate}[(i)]
\item non-degenerate idempotent ($k$-projective) $B$-extension structures on $A$;
\item algebra maps $f:B\to \MM(A)$ such that $A$ is a non-degenerate idempotent ($k$-projective) $B$-bimodule with $B$-actions induced by $f$;
\item functors $F$, $F_r$ and $F_\ell$ which render commutative the following diagram of functors (where the unlabeled arrows are forgetful functors); 
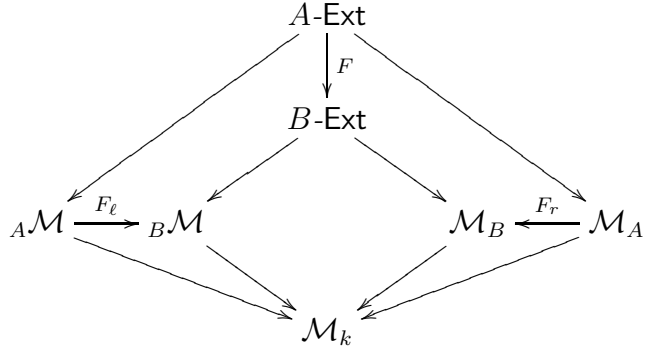
\[
\xymatrix{
&&A\hbox{-}\Ext \ar[ddll] \ar[ddrr] \ar[d]^{F}\\
&&B\hbox{-}\Ext \ar[dl] \ar[dr] \\
{_A\Mm}\ar[r]^-{F_\ell} \ar[drr] &{_B\Mm} \ar[dr] && \Mm_B \ar[dl] & \Mm_A \ar[l]_-{F_r} \ar[dll]\\
&& \Mm_k
}
\]
\item functors $F$ which render commutative the following diagram of functors (where the unlabeled arrows are forgetful functors). 
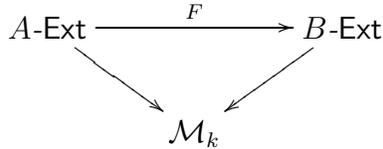
\[
\xymatrix{
A\hbox{-}\Ext\ar[rr]^--{F} \ar[dr] &&B\hbox{-}\Ext \ar[dl] \\
&\Mm_k
}
 \]
\end{enumerate}
In any of the above equivalent situations, the map $f:B\to \MM(A)$ from part $(ii)$ can be extended uniquely to a unital algebra morphism $\bar{f}:\MM(B)\to \MM(A)$ such that $\bar{f}\circ \iota_B = f$.
\end{theorem}

\begin{proof}
$\ul{(i)\Leftrightarrow(ii)}$. This equivalence follows directly from \leref{extvsmod} (iii).\\
$\ul{(ii)\Rightarrow(iii)}$. The functor $F_r$ is constructed in \leref{extvsfunct}; the functor $F_\ell$ is constructed in a symmetrical way. To construct the functor $F$, take a non-degenerate idempotent extension $R$ of $A$. The left and right $B$-action on $R$ are induced by the functors $F_\ell$ and $F_r$. We only have to verify that these actions impose that $R$ is a $B$-bimodule and that the multiplication map $\mu_R$ of $R$ is $B$-bilinear and $B$-balanced. We will prove the $B$-balancedness of $\mu_R$, and leave the other verifications to the reader. Let $r,r'\in R$ and $b\in B$. Since $R$ is an idempotent $A$-bimodule, we can write $r=\sum_ir_ia_i\in RA$ and $r'=\sum_j a'_jr'_j\in AR$. We then have 
\begin{eqnarray*}
&&\hspace{-2cm}
(r\cdot b)r'= \sum_i (r_i(a_i\cdot b))r'= \sum_i r_i((a_i\cdot b)r')
= \sum_{i,j} r_i((a_i\cdot b)(a'_jr'_j))\\
&=& \sum_{i,j} r_i(((a_i\cdot b)a'_j)r'_j)= \sum_{i,j} r_i((a_i (b\cdot a'_j))r'_j)
= \sum_{i,j} r_i(a_i ((b\cdot a'_j)r'_j))\\
&=& \sum_i r_i(a_i (b \cdot r'))= \sum_i (r_ia_i) (b \cdot r')= r(b\cdot r'),
\end{eqnarray*}
as wanted. Here we used in equalities two and eight that $\mu_R$ is $A$-balanced, and in the fifth equality that $\mu_A$ is $B$-balanced.\\
$\ul{(iii)\Rightarrow(iv)}$. Trivial.\\
$\ul{(iv)\Rightarrow(i)}$. Obviously $A$ is an object of $A$-$\Ext$, hence $A=F(A)\in B\hbox{-}\Ext$.\\
The last statement was proven in \cite[Proposition A.5]{VD:Mult}. The map $\bar{f}:\MM(B)\to \MM(A)$ is defined by the following formulas: 
\begin{eqnarray}
\bar{f}(x)\lactd a&=& \sum_i f(x \lactd b_i)\lactd a_i= \sum_i f(\lambda_x(b_i))\lactd a_i, \textrm{ and}\\
a \ractd \bar{f}(x)&=& \sum_j a'_j\ractd f(b'_j \ractd x)= \sum_j a'_j\ractd f(\rho_x(b'_j)),
\end{eqnarray}
where $x=(\lambda_x,\rho_x)\in \MM(B)$, $a\in A$ and $a=\sum_i b_ia_i\in BA$, $a=\sum_j a'_jb'_j\in AB$.
\end{proof}

\subsection{The monoidal category of non-degenerate idempotent $k$-projective algebras}

We can now introduce the category $\ndi_k$ as the category whose objects are non-degenerate idempotent $k$-projective algebras and whose morphisms are non-degenerate idempotent ($k$-projective) extensions (or simply idempotent extensions, in view of \leref{111}). 
The reason for the extra assumption that the algebras are projective as $k$-modules is explained by the following Lemma.
Let $A$ and $B$ be two algebras. Then $A\ot B$ is again an algebra with multiplication
\begin{equation}\eqlabel{tensoralgebra}
(a\ot b)(a'\ot b')= aa'\ot bb',
\end{equation}
for all $a,a'\in A$ and $b,b'\in B$. Similarly, given any right $A$-module $M$ and right $B$-module $N$, $M\ot N$ is a right $A\ot B$-module with action
$$(m\ot n)(a\ot b)= ma\ot nb,$$
for all $m\in M$, $n\in N$, $a\in A$ and $b\in B$.

\begin{lemma}\lelabel{tensor}
Let $A$ and $B$ be two algebras. Suppose that $M$ is a non-degenerate idempotent $k$-projective right $A$-module and $N$ a non-degenerate idempotent $k$-projective right $B$-module. Then $M\ot N$ is a non-degenerate idempotent $k$-projective right $A\ot B$-module.
\end{lemma}

\begin{proof}
The $k$-projectivity and idempotency of $M\ot N$ are obvious. To prove the non-degeneracy, remark that $M$ is non-degenerate as right $A$-module if and only if the map $\phi_M:M\to \Hom(A,M), m\mapsto (a \mapsto ma) $ is injective. Since we assumed that $N$ and $M$ are projective and thus flat as $k$-modules, we then obtain that the maps
$$\phi_M\ot N:M\ot N\to \Hom(A,M)\ot N \quad {\rm and }\quad M\ot \phi_N:M\ot N\to M\ot \Hom(A,N)$$
are both injective. 
If we compose these maps respectively with the injective maps $\beta_{A,M,N}:\Hom(A,M)\ot N\to \Hom(A,M\ot N)$ and $\alpha_{M,A,N}: M\ot \Hom(A,N)\to \Hom(A,M\ot N)$ of \equref{alphabeta}, we get the following injective maps
$$\varphi_{M,N}:M\ot N\to \Hom(A,M\ot N),~ \varphi_{M,N}(m\ot n)(a)=ma \ot n$$
and
$$\psi_{M,N}:M\ot N\to \Hom(A,M\ot N),~ \psi_{M,N}(m\ot n)(a)=m \ot na.$$
Now consider any  $\sum_i m_i\ot n_i\in M\ot N$ and suppose that $(\sum_i m_i\ot n_i) (a\ot b)= \sum_i m_ia\ot n_ib=0$ for all $a\ot b\in A\ot B$. If we keep $a$ fixed and let $b$ vary over $B$, then we obtain by the injectivity of $\psi_{M,N}$ that $\sum_i m_ia\ot n_i = 0$. Since this holds for all $a\in A$, we can now apply the injectivity of $\varphi_{M,N}$, and we find  $\sum_i m_i\ot n_i =0$. Hence $M\ot N$ is non-degenerate as right $A\ot B$-module.
\end{proof}

\begin{proposition}\prlabel{catNdI}
With notation as above, $\ndi_k$ is a monoidal category.
\end{proposition}

\begin{proof}
\mbox{  }
\noindent
Consider two objects $A$ and $B$ in $\ndi_k$. A morphism $f$ from $B$ to $A$ in $\ndi_k$, that is a (non-degenerate) idempotent ($k$-projective) extension $A$ of $B$, will be denoted by
$f:B\mor A$. The vertical bar in the middle of the arrow reminds us to the fact that $f$ is not an actual map.\\
Consider two morphisms $f:B\mor A$ and $g:A\mor R$. Since we have a functor $F:A\hbox{-}\Ext\to B\hbox{-}\Ext$ as in \thref{charext} (iv), and we have that $R\in  A\hbox{-}\Ext$, it follows that $R=F(R)\in  B\hbox{-}\Ext$. We define the composition $g\circ f:B\mor R$ as this (non-degenerate) idempotent ($k$-projective) extension $R$ of $B$.
The identity morphism of $A$ is the trivial non-degenerate idempotent ($k$-projective) extension $A$ of $A$. 
It will be denoted by $A:A\mor A$.\\
Given two objects $A,B\in \ndi_k$, a similar computation as in the proof of \leref{tensor} shows that $A\ot B$ with multiplication as defined in \equref{tensoralgebra} is again an object in $\ndi_k$.
Consider morphisms $f:B\mor A$ and $f':B'\mor A'$ in $\ndi_k$. From \leref{tensor} it follows that $A\ot A'$ is a (non-degenerate) idempotent ($k$-projective) extension of $B\ot B'$. We define $$f\ot f':B\ot B'\mor A\ot A'$$ as this extension.
We leave the other verifications for $\ndi_k$ to be a monoidal category (associativity and unit constraints and coherence conditions) to the reader.
\end{proof}

\begin{notation}
For a morphism $f:B\mor A$ in $\ndi_k$, i.e. a non-degenerate idempotent ($k$-projective) $B$-extension $A$, we denote by $\widetilde{f}:B\to \MM(A)$ the unique algebra map we can associate to it such that $A$ is a non-degenerate idempotent ($k$-projective) $B$-bimodule with $B$-actions induced by $\widetilde{f}$, and by $\ol{f}:\MM(B)\to \MM(A)$ the unique extension of $\widetilde {f}$ to $\MM(B)$ (see \thref{charext}). 
For the identity morphism $A:A\mor A$ on an object $A\in \ndi_k$ we then have that $\widetilde{A}=\iota_A$ and $\ol{A}=\MM(A)$. Note that because of the fact that $A\in \ndi_k$, the algebra map $\iota_A$ indeed induces a non-degenerate idempotent ($k$-projective) $A$-bimodule structure on $A$.
The composition $g\circ f:B\mor R$ of another morphism $g:A\mor R$ in $\ndi_k$ with $f$ is characterized by the algebra map $\widetilde{g\circ f}=\ol{g}\circ \widetilde{f}:B\to \MM(R)$ and its unique extension $\ol{g\circ f}=\ol{g}\circ \ol{f}:\MM(B)\to \MM(R)$ to $\MM(B)$.
Given another morphism $f':B'\mor A'$ in $\ndi_k$, we have for the tensor product $f\ot f':B\ot B'\mor A\ot A'$ in $\ndi_k$ that $\widetilde{f\ot f'}=\Psi_{A,A'}\circ (\widetilde{f}\ot \widetilde{f'}):B\ot B'\to \MM(A\ot A')$, where $\Psi_{A,A'}:\MM(A)\ot \MM(A')\to \MM(A\ot A')$ is the canonical inclusion.
\end{notation}

\begin{remark}
The category $\ndi_k$ can now be thought of as the category whose objects are non-degenerate idempotent $k$-projective algebras, and whose morphisms between two objects $B$ and $A$ are algebra maps $\widetilde{f}:B\to \MM(A)$ that induce a non-degenerate idempotent ($k$-projective) $B$-bimodule structure on $A$.
\end{remark}
\begin{theorem}
We have the following diagram of monoidal functors.
\[
\xymatrix{\ndi_k \ar[rr]^-\MM  && \Alg_k^u 
\ar[rr]^-U && \Mm_k}
\]
Here $\Alg_k^u$ denotes the category of unital $k$-algebras, and $U$ is a forgetful functor.
Moreover, if $k$ is a field, then we have an adjoint pair $(\MM,U_\MM)$ of functors, where $U_\MM: \Alg_k^u\to \ndi_k$ is a forgetful functor that is fully faithful.

\end{theorem}

\begin{proof}
The functor $\MM$ is defined as follows. For a non-degenerate idempotent $k$-projective algebra $A$, $\MM(A)$ is the multiplier algebra. If $f:B\mor A$ is a non-degenerate idempotent ($k$-projective) $B$-extension, then we know by \thref{charext} that this is equivalent to the (unique) existence of an algebra map $\widetilde{f}:B\to \MM(A)$ that induces a non-degenerate idempotent ($k$-projective) $B$-bimodule structure on $A$, which can moreover be extended to a unital algebra morphism $\ol{f} : \MM(B)\to \MM(A)$. We define $\MM(f)=\ol{f}$. Clearly $\MM$ is a functor.

If $k$ is a field, then any $k$-module is $k$-projective and one can consider the forgetful functor 
$U_\MM: \Alg_k^u\to \ndi_k$.
To see that $(\MM,U_\MM)$ is an adjoint pair, we define the unit $\eta$ and counit $\epsilon$ of the adjunction. For any $A\in \ndi_k$ and $R\in \Alg_k^u$, we define 
$$\eta_A=\iota_A : A\to \MM(A), \qquad \epsilon_R = R : \MM(R)  =R \to R.$$
Since the multiplier algebra of a unital algebra is the original algebra itself, the counit is the identity natural transformation. Hence $U_\MM$ is a fully faithful functor.
Moreover, there is a natural transformation
$$\Psi_{A,B}:\MM(A)\ot \MM(B)\to \MM(A\ot B),$$
which implies that $\MM$ is indeed a monoidal functor.
\end{proof}

\section{Multiplier bialgebras}\selabel{multibi}

\subsection{Multiplier bialgebras} 

\begin{Definition}
A \emph{multiplier $k$-bialgebra} is a comonoid in the monoidal category $\ndi_k$.
\end{Definition}

Let us spend some time on restating this definition in a more explicit way. 
By definition a multiplier bialgebra is a triple $A=(A,\Delta,\varepsilon)$
consisting of a non-degenerate idempotent $k$-projective algebra $A$ and morphisms $\Delta:A\mor A\ot A$ and $\varepsilon:A\mor k$ in $\ndi_k$, such that 
$$(\Delta \ot A)\circ \Delta= (A\ot \Delta)\circ \Delta$$ 
as morphisms from $A$ to $A\ot A\ot A$ in $\ndi_k$ and 
$$(\varepsilon \ot A)\circ \Delta= A = (A\ot \varepsilon)\circ \Delta$$
as morphisms from $A$ to $A$ in $\ndi_k$.
Making use of the algebra maps
$\widetilde{\Delta}:A\to \MM(A\ot A)$ and $\widetilde{\varepsilon}:A\to \MM(k)=k$
these two conditions can be expressed as
\begin{equation}\eqlabel{coass}
\ol{\Delta\ot A}\circ\widetilde{\Delta}=\ol{A\ot\Delta}\circ\widetilde{\Delta} 
\end{equation}
and
\begin{equation}\eqlabel{counit}
\ol{\varepsilon\ot A}\circ \widetilde{\Delta}=\iota_A=\ol{A\ot \varepsilon}\circ \widetilde{\Delta}. 
\end{equation}
Using, for all $b,b'\in A$, the notation $b=\sum b^1b^2$ and $b\ot b'=\sum \widetilde\Delta(\ul{b})\lactd (\ul{b\ot b'})=\sum (\ol{b\ot b'}) \ractd \widetilde\Delta(\ol{b})$ to express the idempotency of $B$ and of the $A$-extension $\Delta$, we can rephrase the coassociativity condition \equref{coass} as
\begin{equation}\eqlabel{coass3}
\sum \widetilde{\Delta\ot A} \big(\widetilde\Delta(a)\lactd  (\ul{b}\ot{b''}^{1}) \big) \lactd (\ul{b\ot b'}\ot {b''}^{2})
=\sum\widetilde{A\ot\Delta} \big(\widetilde\Delta(a)\lactd  (b^1\ot\ul{b'}) \big) \lactd (b^2\ot \ul{b'\ot b''})
\end{equation}
and
\begin{equation}\eqlabel{coass2}
\sum(\ol{b\ot b'}\ot {b''}^{1})\ractd \widetilde{\Delta\ot A}\big((\ol{b}\ot {b''}^{2})\ractd \widetilde\Delta(a)\big)
=\sum(b^1\ot \ol{b'\ot b''})\ractd \widetilde{A\ot\Delta} \big((b^2\ot\ol{b'}) \ractd\widetilde\Delta(a)   \big), 
\end{equation}
where $b\ot b'\ot b''\in A\ot A\ot A$ and $a\in A$.
The left counit condition, i.e. the first equality of \equref{counit}, can also be read as
\begin{equation}\eqlabel{counit2}
\ol{\varepsilon\ot A}\big(\widetilde{\Delta}(a)\big)\lactd b=ab, \qquad b\ractd \ol{\varepsilon\ot A}\big(\widetilde{\Delta}(a)\big)=ba,
\end{equation}
for all $a,b\in A$. This formula can be made even more explicit. Remark first that the fact that the $A$-extension $\varepsilon:A\mor k$ is idempotent means exactly that $\widetilde\varepsilon:A\to \MM(k)=k$ is surjective, i.e. there exists an element $g\in A$ such that $\widetilde\varepsilon(g)=1_k$.  
For any $b\in A$, we can write $b=\sum \widetilde\varepsilon(g)b^1b^2=\sum \widetilde{\varepsilon\ot A}(g\ot b^1)\lactd b^2=\sum b^1\ractd \widetilde{A\ot\varepsilon}(b^2\ot g)$, the last two sums being two ways to express $b$ due to the idempotency of the extension $\varepsilon \ot A$. So \equref{counit2} becomes
\begin{equation}\eqlabel{counit3}
\widetilde{\varepsilon\ot A}\big(  \widetilde{\Delta}(a)\lactd (g\ot b^1) \big) \lactd b^2 = ab, \qquad
b^1 \ractd \widetilde{A\ot \varepsilon}\big( (b^2\ot g) \ractd   \widetilde{\Delta}(a)  \big)  = ba.
\end{equation}

Our definition of multiplier bialgebra is closely related to the one introduced by Van Daele in \cite{VD:Mult}. We will make this relationship more explicit in \seref{multihopf}, but we already show in the next Propositions that our notions of coassociativity and counitality coincide with those of \cite{VD:Mult}.
Before we state and prove these, we first introduce some notation and prove a Lemma. 

Given algebras $A_1,\ldots,A_n$, consider the multiplier algebras $\MM(A_1),\ldots,\MM(A_n)$ and $\MM(A_1\ot \cdots \ot A_n)$. Denote by $\Psi_{A_1,\ldots,A_n}:\MM(A_1)\ot \cdots \ot \MM(A_n) \to \MM(A_1\ot \cdots \ot A_n)$ the linear map defined by
\begin{eqnarray*}
\Psi_{A_1,\ldots,A_n}(x_1\ot \cdots \ot x_n)\lactd (a_1\ot \cdots \ot a_n)&=& (x_1\lactd a_1)\ot \cdots \ot (x_n\lactd a_n),\\
(a_1\ot \cdots \ot a_n)\ractd \Psi_{A_1,\ldots,A_n}(x_1\ot \cdots \ot x_n)&=& (a_1\ractd x_1) \ot \cdots \ot (a_n\ractd x_n),
\end{eqnarray*}
where $a_i\in A_i$ and $x_i\in \MM(A_i)$, for all $i=1,\ldots,n$. 
A multiplier of the form $\Psi_{A_1,\ldots,A_n}(1\ot\cdots \ot 1\ot \iota_{A_i}(a_i)\ot 1\ot \cdots \ot 1)\in \MM(A_1\ot \cdots \ot A_n)$ will be denoted shortly by $1\ot\cdots \ot 1\ot a_i\ot 1\ot \cdots \ot 1$, where $\iota_{A_i}(a_i)$ and $a_i$ appear in the $i$th tensorand.
\begin{lemma}\lelabel{help}
Let $f:A\mor B$ be a morphism in $\ndi_k$, $C\in \ndi_k$ and $c\in C$.\\
For any $x\in \MM(A\ot C)$ we have the following equality in $\MM(B\ot C)$.
\begin{eqnarray}
\ol{f\ot C}(x(1\ot c))&=& \ol{f\ot C}(x)(1\ot c)\eqlabel{help1}
\end{eqnarray}
For any $x\in \MM(C\ot A)$ we have the following equality in $\MM(C\ot B)$.
\begin{eqnarray}
\ol{C\ot f}((c\ot 1)x)&=& (c\ot 1)\ol{C\ot f}(x)\eqlabel{help2}
\end{eqnarray}
\end{lemma}

\begin{proof}
We prove the first equality, by proving the equality of the left multipliers of both sides; the proof for the right multipliers is completely analogous. For all $b\in B$ and $c,d\in C$ we have that
\begin{eqnarray*}
&&\hspace{-1cm}
\ol{f\ot C}(x(1\ot c))\lactd (b\ot d)
= \sum_{i} \widetilde{f \ot C}((x(1\ot c))\lactd (a_i\ot d^1))\lactd (b_i \ot d^2)\\
&=& \sum_{i} \widetilde{f \ot C}(x\lactd (a_i\ot cd^1))\lactd (b_i\ot d^2)= \ol{f\ot C}(x)\lactd (b \ot cd)\\
&=& \big( \ol{f\ot C}(x)(1\ot c)\big)\lactd(b\ot d),
\end{eqnarray*}
where $b=\sum_i\widetilde{f}(a_i)\lactd b_i\in f(A)\lactd B$ and $d=\sum d^1d^2\in C^2$, and thus also $cd=\sum (cd^1)d^2$.
\end{proof}

\begin{proposition}\prlabel{coass} 
A morphism $\Delta:A\mor A\ot A$ in $\ndi_k$ is coassociative in the sense of \equref{coass} if and only if $\widetilde{\Delta}:A\to \MM(A\ot A)$ is coassociative in the sense of \cite{VD:Mult}, i.e. 
\begin{eqnarray}\eqlabel{coassfonzie}
(a\ot 1\ot 1)\ol{\Delta \ot A}(\widetilde{\Delta}(b)(1\ot c))
&=&  \ol{A\ot \Delta}((a\ot 1)\widetilde{\Delta}(b))(1\ot 1\ot c),
\end{eqnarray}
for all $a,b,c\in A$.
\end{proposition}
\begin{proof}
Suppose that \equref{coass} holds, then
\begin{eqnarray*}
&&\hspace{-1.5cm}
(a\ot 1\ot 1)\ol{\Delta \ot A}(\widetilde{\Delta}(b)(1\ot c))
\equal{\equref{help1}}
(a\ot 1\ot 1)\big(\ol{\Delta \ot A}(\widetilde{\Delta}(b))(1\ot 1\ot c)\big)\\
&=&\big( (a\ot 1\ot 1)\ol{A\ot \Delta}(\widetilde{\Delta}(b))\big) (1\ot 1\ot c) \equal{\equref{help2}}
 \ol{A\ot \Delta}((a\ot 1)\widetilde{\Delta}(b))(1\ot 1\ot c).
\end{eqnarray*}
Conversely, suppose that \equref{coassfonzie} holds. It is easy to check (with methods similar to the ones in the proof of \leref{tensor}) that, by the non-degeneracy of $A\ot A\ot A$, for any $x\in A\ot A\ot A$ we have that $x=0$ if $(a\ot 1\ot 1)\lactd x=0$, for all $a\in A$.
So, to prove that the left multipliers of $\ol{ \Delta\ot A}(\widetilde{\Delta}(a))$ and $\ol{A\ot \Delta}(\widetilde{\Delta}(a))$ are equal for all $a\in A$, it suffices to prove that, for all $a',b,c,d\in A$, we have that 
\begin{eqnarray*}
&&\hspace{-2.5cm}
\big((a'\ot 1\ot 1) \ol{ \Delta\ot A}(\widetilde{\Delta}(a))\big)\lactd (b\ot c\ot d) \\
&=&
\sum\big( (a'\ot 1\ot 1) \ol{\Delta \ot A}(\widetilde{\Delta}(a)) (1\ot 1\ot d^1) \big)\lactd (b\ot c\ot d^2)\\
&\equal{\equref{help1}}&  
\sum\big( (a'\ot 1\ot 1)\ol{\Delta \ot A}(\widetilde{\Delta}(a)(1\ot d^1)) \big)\lactd (b\ot c\ot d^2)
\end{eqnarray*}
equals 
\begin{eqnarray*}
&&\hspace{-2cm}
\big((a'\ot 1\ot 1) \ol{A\ot \Delta}(\widetilde{\Delta}(a))\big)\lactd (b\ot c\ot d)
\equal{\equref{help2}} \ol{A\ot \Delta}((a'\ot 1)\widetilde{\Delta}(a))\lactd (b\ot c\ot d)\\
&=& \sum\big( \ol{A\ot \Delta}((a'\ot 1)\widetilde{\Delta}(a))(1\ot 1\ot d^1) \big)\lactd (b\ot c\ot d^2),
\end{eqnarray*}
and this is the case by \equref{coassfonzie}. Here we used again the notation $d=\sum d^1d^2\in A^2$.
\end{proof}
\begin{proposition}\prlabel{counit}
A morphism $\Delta:A\mor A\ot A$ in $\ndi_k$ is counital in the sense of \equref{counit} if and only if $\widetilde{\Delta}:A\to \MM(A\ot A)$ is counital in the sense of \cite{VD:Mult}, i.e. 
\begin{eqnarray}\eqlabel{counitfonzie}
\ol{\varepsilon \ot A}(\widetilde{\Delta}(a)(1\ot b))=\iota_A(ab)= 
\ol{A\ot \varepsilon}((a\ot 1)\widetilde{\Delta}(b)),
\end{eqnarray}
for all $a,b\in A$.
\end{proposition}
\begin{proof}
Suppose that \equref{counit} holds. We prove the first equality of \equref{counitfonzie}:
\begin{eqnarray*}
&&\hspace{-1.5cm}
\ol{\varepsilon \ot A}(\widetilde{\Delta}(a)(1\ot b))\equal{\equref{help1}} 
\ol{\varepsilon \ot A}(\widetilde{\Delta}(a))\iota_{A}(b)\equal{\equref{counit}} \iota_A(a)\iota_A(b)=\iota_A(ab),
\end{eqnarray*}
where we used that $\MM(k\ot A)\cong \MM(A)$ in the first equality.
Conversely, suppose that we are given \equref{counitfonzie}. Then for all  $a,a',b\in A$ we have 
\begin{eqnarray*}
&&\hspace{-1.5cm}
\big(\ol{\varepsilon \ot A}(\widetilde{\Delta}(a))\lactd b\big)a'
= \ol{\varepsilon \ot A}(\widetilde{\Delta}(a))\lactd (ba')
= \big(\ol{\varepsilon \ot A}(\widetilde{\Delta}(a))\iota_{A}(b)\big)\lactd a'\\
&\equal{\equref{help1}}& \ol{\varepsilon \ot A}(\widetilde{\Delta}(a)(1\ot b))\lactd a'
\equal{\equref{counitfonzie}} \iota_A(ab)\lactd a'= (ab)a'=(\iota_A(a)\lactd b)a',
\end{eqnarray*}
such that by the non-degeneracy of $A$ it follows that $\ol{\varepsilon \ot A}(\widetilde{\Delta}(a))\lactd b
= \iota_A(a)\lactd b$, for all $a,b\in A$. By a similar computation, one shows that the right multipliers of $\ol{\varepsilon\ot A}(\widetilde\Delta(a))$ and $\iota_A(a)$ are equal.
\end{proof}

\subsection{Monoidal structures on module categories}

In this Section we will show that a non-degenerate idempotent $k$-projective algebra $A$ is a multiplier bialgebra if and only if 
the category of its non-degenerate idempotent $k$-projective algebra extensions and both the categories of its non-degenerate idempotent $k$-projective left and right modules are monoidal and fit, together with the category of $k$-modules, into a diagram of strict monoidal forgetful functors (in the sense of \cite{McLane}). 

\begin{lemma}\lelabel{Amodtensor}
Let $A$ be a non-degenerate idempotent ($k$-projective) algebra and $\Delta:A\mor A\ot A$ a non-degenerate idempotent ($k$-projective) extension. For any two $M, N\in \Mm_A$, we have $M\ot N\in\Mm_A$ with right $A$-action defined by 
\begin{equation}\eqlabel{AactonMotN}
(m\ot n)\cdot a := \sum_{i,j} (m_i\ot n_j)((a_i\ot b_j)\ractd \widetilde{\Delta}(a)),
\end{equation}
for all $a\in A,m\in M,n\in N$ and where $m=\sum_i m_ia_i\in MA$ and $n=\sum_j n_jb_j\in NA$.
\end{lemma}

\begin{proof}
By \leref{tensor} $M\ot N$ is a non-degenerate idempotent $k$-projective right $A\ot A$-module. \prref{multialg} (vii) then implies that $M\ot N$ is a unital right $\MM(A\ot A)$-module. \equref{AactonMotN} means nothing else than
$(m\ot n)\cdot a=(m\ot n)\ractd \widetilde{\Delta}(a)$, where $\ractd$ is exactly the aforementioned right $\MM(A\ot A)$-action on $M\ot N$. Hence the right $A$-action on $M\ot N$ is well-defined and associative.
Let us prove that this action is also non-degenerate. 
Recall that since $\Delta$ is an idempotent extension, given $a\ot a'\in A\ot A$, we can find elements $a_i\in A$ and $b_i\ot b'_i\in A\ot A$ such that $a\ot a'=\sum_i \widetilde{\Delta}(a_i)\lactd(b_i\ot b'_i)$. Now take any $m\ot n\in M\ot N$, then we find
\begin{eqnarray*}
(m\ot n)(a\ot a')&\equal{\equref{hulp}}&(m\ot n)\ractd\iota_{A\ot A}(a\ot a')
=\sum_i(m\ot n)\ractd\iota_{A\ot A}(\widetilde{\Delta}(a_i)\lactd(b_i\ot b'_i))\\
&&\hspace{-2cm}\equal{\equref{AbimMA}}\sum_i(m\ot n)\ractd(\widetilde{\Delta}(a_i)\iota_{A\ot A}(b_i\ot b'_i))
=\sum_i((m\ot n)\ractd\widetilde{\Delta}(a_i))\ractd\iota_{A\ot A}(b_i\ot b'_i)\\
&&\hspace{-2cm}=\sum_i((m\ot n)\cdot a_i)\ractd\iota_{A\ot A}(b_i\ot b'_i)\equal{\equref{hulp}} \sum_i((m\ot n)\cdot a_i)(b_i\ot b'_i).\\
\end{eqnarray*}
Suppose now that $(m \ot n)\cdot a=0$, for all $a\in A$. Then we know, using the above computation, that for all $a\ot a'\in A\ot A$, $(m\ot n)(a\ot a')=\sum_i((m\ot n)\cdot a_i)(b_i\ot b'_i)=0$. Since we know already that $M\ot N$ is non-degenerate as right $A\ot A$-module, we find that $m\ot n=0$, hence $M\ot N$ is also non-degenerate as right $A$-module.

Finally, let us verify that the action of $A$ on $M\ot N$ is idempotent. This follows from 
\begin{eqnarray*}
&&\hspace{-2cm}
M\ot N = (M\ot N)(A\ot A)= (M\ot N)((A\ot A) \ractd \widetilde{\Delta}(A))\\
&\equal{\equref{HULP}}& ((M\ot N)(A\ot A)) \ractd \widetilde{\Delta}(A) 
= (M\ot N) \ractd \widetilde{\Delta}(A)
= (M\ot N)\cdot A,
\end{eqnarray*}
where we used in the first and the fourth equality the idempotency of $M\ot N$ as right $A\ot A$-module and in the second one the idempotency of $A\ot A$ as right $A$-module (with $A$-action induced by $\widetilde{\Delta}$, the algebra map associated to the idempotent extension $\Delta$).
\end{proof}

Consider again a non-degenerate idempotent $k$-projective algebra $A$. Suppose that the following is a diagram of monoidal categories and strict monoidal forgetful functors. 
\begin{equation}\eqlabel{forgetA}
\xymatrix{
&A\hbox{-}\Ext \ar[dl] \ar[dr] \\
{_A\Mm} \ar[dr] && \Mm_A \ar[dl]\\
& \Mm_k
}
\end{equation}
Since $A$ is an object in $A$-$\Ext$, $A\ot A$ is also an object in $A$-$\Ext$, i.e. there is a non-degenerate idempotent ($k$-projective) extension $\Delta:A\mor A\ot A$. Recall from \thref{charext} that $\Delta$ induces an algebra map $\widetilde\Delta:A\to \MM(A\ot A)$ by putting 
\begin{equation}\eqlabel{tildedelta}
\widetilde\Delta(b)\lactd (a\ot a')=b\cdot (a\ot a'), \qquad (a\ot a')\ractd \widetilde\Delta(b)=(a\ot a')\cdot b,
\end{equation}
for all $a\ot a'\in A\ot A$ and $b\in A$.
Then by \leref{Amodtensor} we know that for any two $M,N\in \Mm_A$, we have an $A$-module structure on $M\ot N$. The following Lemma asserts that this $A$-module structure on $M\ot N$ coincides with the $A$-module structure given by the monoidal structure on $\Mm_A$.

\begin{lemma}\lelabel{compatibleAmod}
Let $A$ be a non-degenerate idempotent $k$-projective algebra such that \equref{forgetA} is a diagram of monoidal categories and strict monoidal forgetful functors. Then for any two $M, N\in\Mm_A$, we have with notation as above
\begin{equation}\eqlabel{compatibleAmod}
(m\ot n)\cdot a = \sum_{i,j} (m_i\ot n_j)((a_i\ot b_j)\ractd \widetilde{\Delta}(a)),
\end{equation}
for all $a\in A,m\in M$ and $n\in N$, where $m=\sum_i m_ia_i\in MA$ and $n=\sum_j n_jb_j\in NA$.
\end{lemma}

\begin{proof}
Because of the idempotency as $A\ot A$-module, it suffices to prove that for all $m\ot n\in M\ot N$, $a\ot a'\in A\ot A$ and $b\in A$,
$$\big((m\ot n)(a\ot a')\big)\cdot b = (m\ot n) \big((a\ot a')\ractd \widetilde\Delta(b)\big).$$
Consider the morphisms $\psi: A\to M,\ \psi(a)=ma$ and $\phi:A\to N,\ \phi(a)=na$ in $\Mm_A$. Since $\Mm_A$ is a monoidal category, we know that $\psi\ot\phi: A\ot A\to M\ot N,\ (\psi\ot\phi)(a\ot a')=ma\ot na'=(m\ot n)(a\ot a')$ is a right $A$-linear map. Hence,
$$
((m\ot n)(a\ot a'))\cdot b = ((\psi\ot\phi)(a\ot a'))\cdot b 
= (\psi\ot\phi)((a\ot a') \ractd \widetilde\Delta(b)) = \\ 
(m\ot n) \big((a\ot a')\ractd \widetilde\Delta(b)\big), 
$$
where we used the right $A$-linearity of $\psi\ot\phi$ together with \equref{tildedelta} in the second equality.
\end{proof}

\begin{lemma}\lelabel{Deltacoass}
Let $A$ be a non-degenerate idempotent $k$-projective algebra such that \equref{forgetA} is a diagram of monoidal categories and strict monoidal forgetful functors. Then the non-degenerate idempotent ($k$-projective) extension $\Delta$ as defined above is coassociative in the sense of \equref{coass}.
\end{lemma}

\begin{proof}
Since $A\in \Mm_A$, $\Mm_A$ is a monoidal category and the forgetful functor to $\Mm_k$ is strict monoidal, we know that the map $f:A\ot(A\ot A)\to (A\ot A)\ot A,\ f(a\ot (a'\ot a''))=(a\ot a')\ot a''$ is an isomorphism of right $A$-modules. Furthermore, by \leref{compatibleAmod} the right $A$-module structure on these isomorphic right $A$-modules can be computed using the formula \equref{compatibleAmod}. Performing the computation of $f((b\ot (b'\ot b'')) \cdot a) = ((b\ot b')\ot b'')\cdot a$ explicitly results exactly in the formula \equref{coass2}. By a left-right symmetric argument, using the monoidal structure on ${_A\Mm}$, we find that also formula \equref{coass3} holds, and therefore $\Delta$ is coassociative.
\end{proof}

Suppose, as above, that $A$ is a non-degenerate idempotent $k$-projective algebra such that \equref{forgetA} is a diagram of monoidal categories and strict monoidal forgetful functors. Then, in particular, $k\in A\hbox{-}\Ext$, and therefore, by \thref{charext}, there is a non-degenerate idempotent extension $\varepsilon: A\mor k$, determined by $\widetilde\varepsilon:A\to \MM(k)= k:a \mapsto 1_k\cdot a = a\cdot 1_k$. Recall also that because of the idempotency of $\varepsilon$, $\widetilde \varepsilon$ is surjective. We denote by $g\in A$ a fixed element such that $\widetilde\varepsilon(g)=1_k$.

\begin{lemma}\lelabel{Deltacounit}
Let $A$ be a non-degenerate idempotent $k$-projective algebra such that \equref{forgetA} is a diagram of monoidal categories and strict monoidal forgetful functors. Then the non-degenerate idempotent ($k$-projective) extension $\Delta$ as defined above is counital in the sense of \equref{counit}.
\end{lemma}

\begin{proof}
By assumption, $k$ is the monoidal unit of the monoidal category $\Mm_A$, and for any $M\in \Mm_A$ the map $r_M: M\ot k\to M,\ r_M(m\ot t)=mt$ for $m\in M$ and $t\in k$, is an isomorphism of right $A$-modules.
In particular we have that, for all $a,b\in A$
$$r_A((b \ot 1_k)\cdot a) = r_A(b\ot 1_k) a= ba.$$ 
By \equref{compatibleAmod} we find
\begin{eqnarray*}
(b \ot 1_k)\cdot a &=& (b^1\ot 1_k) \big((b^2\ot g)\ractd \widetilde{\Delta}(a)\big)= b^1 \big( (A\ot \widetilde\varepsilon)((b^2\ot g)\ractd \widetilde{\Delta}(a))\big).
\end{eqnarray*}
In the second equality we used the formula for $\widetilde\varepsilon$ given above this Lemma. If we apply $r_A$ to the last expression we obtain exactly the right hand side of \equref{counit3}. Similarly, the left hand side of \equref{counit3} is obtained, using the isomorphism $k\ot A\cong A$ in the monoidal category ${}_A\Mm$ of left $A$-modules.
\end{proof}

We now arrive at the main result of this Section.

\begin{theorem}\thlabel{catint}
Let $A$ be a non-degenerate idempotent $k$-projective algebra, then there is a bijective correspondence between structures of a multiplier bialgebra on $A$ and structures of monoidal categories on $\Mm_A$, ${}_A \Mm$ and $A$-$\Ext$ such that all forgetful functors in diagram \equref{forgetA} are strict monoidal.
\end{theorem}

\begin{proof}
Suppose first that \equref{forgetA} is a diagram of monoidal categories and strict monoidal forgetful functors. Then by \leref{Deltacoass} and \leref{Deltacounit}, there are (non-degenerate) idempotent algebra ($k$-projective) extensions $\Delta:A\mor A\ot A$ and $\varepsilon:A\mor k$ which turn $A$ into a multiplier bialgebra.

Conversely, if $A$ is a multiplier bialgebra, with comultiplication $\Delta$ and counit $\varepsilon$, then we know by \leref{Amodtensor} that there is a well-defined functor 
$\ot:\Mm_A\times \Mm_A\to \Mm_A$. Using the coassociativity of $\Delta$, we can prove by similar arguments as in \leref{Deltacoass} that, for all $M, N, P\in \Mm_A$ (resp.\ in ${_A\Mm}$), the isomorphism $M\ot(N\ot P)\cong (M\ot N)\ot P$ holds.

Since $\varepsilon:A\mor k$ is a non-degenerate idempotent $k$-projective extension, we know that $k\in A\hbox{-}\Ext$, so in particular $k$ is a non-degenerate idempotent $k$-projective left and right $A$-module. Let us check that $k$ is a monoidal unit in $\Mm_A$. To prove that the map $r_M: M\ot k\to M,\ r_M(m\ot t)=mt$ for $m\in M$ and $t\in k$, is an isomorphism of right $A$-modules, we can use the same (converse) reasoning as in \leref{Deltacounit}. Next, we will show that that the $k$-linear isomorphism $\ell_M:k\ot M\to M,\ \ell_M(t\ot m)=tm$ is also right $A$-linear. Take any $m\in M$ and $a\in A$. Using the idempotency of $M$ as right $A$-module, we can write $m=\sum_i m_ia_i\in MA$. As before, take $g\in A$ such that $\widetilde\varepsilon(g)=1_k$. Then we have $m=\sum_i m_i \ell_A\big( (\widetilde\varepsilon\ot A)  (g\ot a_i)\big) $. Furthermore, using formula \equref{AactonMotN}, we find
\begin{eqnarray*}
\ell_M((1_k\ot m)\cdot a)&=& \sum_i \ell_M\big( (1_k\ot m_i) ((g\ot a_i)\ractd \widetilde{\Delta}(a)) \big)\\
&=& \sum_i m_i \ell_A\big((\widetilde\varepsilon\ot A)((g\ot a_i)\ractd \widetilde{\Delta}(a))\big).
\end{eqnarray*}
If we multiply the last expression with an arbitrary $b=\sum b^1b^2\in A=A^2$, we further obtain
\begin{eqnarray*}
&&\hspace{-2cm}\sum_i m_i \ell_A\Big((\widetilde\varepsilon\ot A)\big((g\ot a_i)\ractd \widetilde{\Delta}(a)\big)\Big) b \\
&&=
\sum_i m_i \ell_A\Big((\widetilde\varepsilon\ot A)\big((g\ot a_i)\ractd \widetilde{\Delta}(a)\big)\Big)\Big(\ell_A\big((\widetilde\varepsilon\ot A)(g\ot b^1) \big)b^2\Big)\\
&&=
\sum_i m_i \Big(\ell_A\big((\widetilde\varepsilon\ot A)\big((g\ot a_i)\ractd \widetilde{\Delta}(a)\big)\big)\ell_A\big((\widetilde\varepsilon\ot A)(g\ot b^1) \big)\Big)b^2\\
&&=
\sum_i m_i \big(\ell_A\circ (\widetilde\varepsilon\ot A)\big)\Big(\big((g\ot a_i)\ractd \widetilde{\Delta}(a)\big) (g\ot b^1)\Big)b^2\\
&&=
\sum_i m_i \big(\ell_A\circ (\widetilde\varepsilon\ot A)\big)\Big( (g\ot a_i) \big(\widetilde{\Delta}(a)\lactd (g\ot b^1)\big)\Big)b^2\\
&&=
\sum_i m_i \ell_A\big( (\widetilde\varepsilon\ot A)  (g\ot a_i)\big) 
\ell_A \Big( (\widetilde\varepsilon\ot A)\big(\widetilde{\Delta}(a)\lactd (g\ot b^1)\big)\Big)b^2\\
&&= m ab = \ell_M(1_k\ot m)ab.
\end{eqnarray*}
Here we used the idempotency of $A$ and the extension $\varepsilon$ in the first equality, the associativity of the $A$-action on $M$ in the second equality, the multiplicativity of the morphisms $\ell_A$ and $\widetilde\varepsilon\ot A$ in the third and penultimate equality, the multiplier property \equref{compatlambdarho} in the fourth equality, and \equref{counit3} in the last equality. 
From the fact that $M$ is non-degenerate as right $A$-module it follows then that $\ell_M((1_k\ot m)\cdot a)=\ell_M(1_k\ot m)a$, as desired.

The verification of the coherence conditions of the associativity and unit constraints is left to the reader.

This completes the proof that $\Mm_A$ is a monoidal category and the forgetful functor to $\Mm_k$ is a strict monoidal functor. Similarly, one proves that ${_A\Mm}$ is a monoidal category. Finally, having two objects $P,Q\in A\hbox{-}\Ext$ we can construct $P\ot Q\in A\hbox{-}\Ext$, by using the left $A$-module structure as computed in ${_A\Mm}$, the right $A$-module structure as in $\Mm_A$ and the $k$-algebra structure as in \equref{tensoralgebra}. It is then easy to verify that $A$-$\Ext$ becomes a monoidal category so that all forgetful functors in diagram \equref{forgetA} are strict monoidal.

To end the proof, we have to show that both constructions are mutually inverse. 
Starting with a multiplier bialgebra, it is an immediate consequence of our constructions that the reconstructed comultiplication and counit coincide with the original ones.
Conversely, if we start with the monoidal structures, construct the multiplier bialgebra and reconstruct the monoidal structures, it follows from \leref{compatibleAmod} that we recover the original monoidal structures. 
\end{proof}

\begin{remark}
If $A$ is an algebra with unity, then \equref{forgetA} will be a diagram of monoidal categories and strict monoidal functors if and only if $\Mm_A$ is a monoidal category and the forgetful functor $\Mm_A\to\Mm_k$ is a strict monoidal functor if and only if ${_A\Mm}$ is a monoidal category and the forgetful functor ${_A\Mm}\to\Mm_k$ is a strict monoidal functor. Indeed, 
since there are algebra isomorphisms $A\cong L(A)\cong R(A)^{\rm op}\cong \MM(A)$, the data of (i), (ii) and (iii) of \leref{extvsmod} are in bijective correspondence in the unital case, from which this statement can be easily derived.
\end{remark}

\section{Multiplier Hopf algebras}\selabel{multihopf}

In this Section we recall the definition of a multiplier Hopf algebra by A.\ Van Daele. We show that the antipode of a multiplier Hopf algebra $A$ can be interpreted as a convolution inverse in a certain ``multiplicative structure'' associated to the underlying multiplier bialgebra of $A$.

\subsection{Van Daele's definition}
Let us first introduce the following Sweedler-type notation. Let $A$ be a non-degenerate algebra and $\widetilde\Delta:A\to \MM(A\ot A)$ an algebra map. Suppose that $\widetilde\Delta$ satisfies the following conditions for all $a, b\in A$: 
\begin{eqnarray}
\widetilde{\Delta}(a)(1\ot b)&\subseteq& \iota_{A\ot A}(A\ot A),\eqlabel{fons1}\\
(a\ot 1)\widetilde{\Delta}(b)&\subseteq& \iota_{A\ot A}(A\ot A).\eqlabel{fons2}
\end{eqnarray}
Given elements $a,b\in A$, we denote by $a_{(1,b)}\ot a_{(2,b)}\in A\ot A$ (summation implicitly understood) the element such that $\widetilde\Delta(a)(1\ot b)=\iota_{A\ot A}(a_{(1,b)}\ot a_{(2,b)})$. Similarly, we denote by $a_{(b,1)}\ot a_{(b,2)}$ the element in $A\ot A$ such that $(b\ot 1)\widetilde\Delta(a)=\iota_{A\ot A}(a_{(b,1)}\ot a_{(b,2)}).$

Recall from \cite{VD:Mult} the following definitions. 
The algebra map $\widetilde\Delta:A\to \MM(A\ot A)$ is called \emph{coassociative}
if the following property holds:
\begin{eqnarray}\eqlabel{coassVD}
(a\ot 1\ot 1)(\widetilde\Delta \ot A)\big(\widetilde\Delta(b)(1\ot c)\big)
&=&(A\ot \widetilde\Delta)\big((a\ot 1)\widetilde\Delta(b)\big)(1\ot 1\ot c),\ {\rm i.e.}\\
b_{(1,c)(a,1)}\ot b_{(1,c)(a,2)} \ot b_{(2,c)} &=& b_{(a,1)}\ot b_{(a,2)(1,c)}\ot b_{(a,2)(2,c)},\nonumber
\end{eqnarray}
for all $a,b, c\in A$ (to be able to let $\widetilde\Delta\ot A$ (resp. $A\ot\widetilde\Delta$) act on $\widetilde\Delta(b)(1\ot c)$ (resp. $(a\ot 1)\widetilde\Delta(b)$), $A\ot A$ is identified with its image in $\MM(A\ot A)$ under $\iota_{A\ot A}$).

A \emph{multiplier Hopf algebra} is a non-degenerate $k$-projective algebra $A$, equipped with a coassociative algebra map $\widetilde\Delta:A\to \MM(A\ot A)$ that satisfies \equref{fons1} and \equref{fons2}, and such that the following maps are bijective:
\begin{eqnarray*}
T_1:A\ot A\to A\ot A,\ && T_1(a\ot b)= a_{(1,b)}\ot a_{(2,b)};\\
T_2:A\ot A\to A\ot A,\ && T_2(a\ot b)=b_{(a,1)}\ot b_{(a,2)}.
\end{eqnarray*}

Van Daele proves that the comultiplication of a multiplier Hopf algebra $A$ is a non-degenerate algebra map in the sense of \reref{nondegenerateVD}.
Furthermore, $A$ can be endowed with an algebra map $\widetilde\varepsilon:A\to k$ that satisfies the following counit conditions:
\begin{eqnarray}
(A\ot \widetilde\varepsilon)\big((a\ot 1)\widetilde\Delta(b)\big) &=& b_{(a,1)}\widetilde\varepsilon(b_{(a,2)}) = ab,\eqlabel{counitVD1}\\ 
(\widetilde\varepsilon\ot A)\big(\widetilde\Delta(a)(1\ot b)\big) &=& \widetilde\varepsilon(a_{(1,b)})a_{(2,b)} = ab,\eqlabel{counitVD2}
\end{eqnarray}
for all $a,b\in A$ (see \cite[Theorem 3.6]{VD:Mult}).
Moreover, there exists an ``antipode'' map $S:A\to \MM(A)$ that satisfies the following properties:
\begin{eqnarray}
m_1(S\ot A)\big(\widetilde{\Delta}(a)(1\ot b)\big)=S(a_{(1,b)})\lactd a_{(2,b)} =\widetilde\varepsilon(a)b,\eqlabel{antipode2},\\
m_2(A\ot S)\big((a\ot 1)\widetilde{\Delta}(b)\big)=b_{(a,1)}\ractd S(b_{(a,2)}) =a\widetilde\varepsilon(b),\eqlabel{antipode1}
\end{eqnarray}
for all $a,b\in A$ (see \cite[Theorem 4.6]{VD:Mult}), where $m_1:\MM(A)\ot A\to A$ and $m_2:A\ot \MM(A)\to A$ are the natural evaluation maps. 
Finally, by \cite[Proposition 1.2]{VDZha:corep} it follows that $A$ is an algebra with local units, so in particular $A$ is idempotent. Remark that in \cite{VD:Mult} $k$ is supposed to be a field. Therefore, $\widetilde\varepsilon$ is automatically surjective (i.e.\ non-degenerate in the sense of \reref{nondegenerateVD}).

We now easily arrive at the following.
\begin{proposition}
A multiplier Hopf algebra is a multiplier bialgebra. 
\end{proposition}
\begin{proof}
By the observations made above, we know that $A$ is a non-degenerate idempotent $k$-projective algebra, equipped with non-degenerate algebra maps $\widetilde\Delta:A\to \MM(A\ot A)$ and $\widetilde\varepsilon: A\to k=\MM(k)$. So by \reref{nondegenerateVD} $A\ot A$ and $k$ are non-degenerate idempotent ($k$-projective) extensions of $A$, i.e.  
$\widetilde\Delta$ and $\widetilde\varepsilon$ give rise to morphisms $\Delta:A\mor A\ot A$ and $\varepsilon: A\mor k$ in $\ndi_k$. The equalities \equref{coassVD}, \equref{counitVD1} and \equref{counitVD2} hold if and only if they hold after applying the injective algebra maps $\iota_{A \ot A\ot A}$ and $\iota_{A}$. Then $\Delta$ is a coassociative and counital comultiplication by \prref{coass} and \prref{counit}. This shows that $A$ is a multiplier bialgebra.
\end{proof}

An immediate consequence of the previous Proposition is that for a multiplier Hopf algebra $A$ we have a diagram of monoidal categories and strict monoidal forgetful functors as in \equref{forgetA}.
\begin{remark}
The definition of a multiplier Hopf algebra differs from the classical definition of a Hopf algebra, not only in its range of generality, but also in its initial set-up. Classically, a Hopf algebra is defined as a bialgebra having an antipode. As we have seen above, a multiplier Hopf algebra is a multiplier bialgebra that possesses an antipode, but the converse is not true. In fact, for a multiplier Hopf algebra the maps $T_1$ and $T_2$ are supposed to be bijective. In case of an algebra $A$ with unit, these conditions are equivalent to $A$ being a Hopf algebra, as we will see in what follows. Recall that for any (usual) bialgebra $A$, one can develop the theory of Hopf-Galois extensions over $A$. In this theory, we study (right) $A$-comodule algebras $B$, and define the coinvariants of $B$ as $B^{{\rm co} A}=\{b\in B~|~ \rho(b)=b\ot 1\}$, where $\rho:B\to B\ot A,~\rho(b)=b_{[0]}\ot b_{[1]}$ is the $A$-coaction on $B$. Then $B^{{\rm co} A} \subset B$ is called a (right) $A$-Galois extension if the canonical map
$$\can:B\ot_{B^{{\rm co}A}} B\to B\ot A,\ \can(b'\ot b)= b'b_{[0]}\ot b_{[1]}$$
is bijective.
A result of Schauenburg \cite{Sch:biHopf} states that a bialgebra $A$ is a Hopf algebra if $A$ itself is a  faithfully flat right $A$-Galois extension. If we consider a bialgebra as right comodule algebra over itself (via its comultiplication map), then $A^{{\rm co}A}\cong k$ and $\can=T_2$. Similarly, 
by considering $A$ as left comodule algebra over itself,
we obtain $T_1$ as the canonical map of the left $A$-Galois extension $k \subset A$. Therefore, we can intuitively understand a multiplier Hopf algebra as being a multiplier bialgebra $A$ such that $k\subset A$ is a left and right ``Hopf-Galois extension''. The construction of the antipode can then be understood as a generalizaton of Schauenburg's result to the non-unital case. Remark that in Van Daele's original definition $k$ is supposed to be a field, so the faithfully flatness condition is automatically satisfied.
\end{remark}

\subsection{The antipode as convolution inverse}
Classically, the antipode of a Hopf algebra $A$ is the inverse of the identity map in the convolution algebra $\End(A)$. 
If $A$ is a multiplier bialgebra, then $\End(A)$ is no longer a (convolution) algebra. In this section we show that it is however possible to put a richer structure on the $k$-module $\Hom(A,\MM(A))$, that makes it possible to define the antipode as a kind of convolution inverse of the map $\iota_A$.

\begin{lemma}
Let $A$ be a non-degenerate idempotent algebra and consider $R=\Hom(A,\MM(A))$. Then 
\begin{enumerate}[(i)]
\item
$R$ is a non-degenerate $A$-bimodule with actions
$$ (b\cdot f \cdot b')(a)=f(b'ab),$$
for all $a,b,b'\in A$ and $f\in R$;
\item
$R$ is a non-degenerate $A$-bimodule with actions
$$(b\lact f \ract b')(a)=\iota_A(b)f(a)\iota_A(b'),$$
for all $a,b,b'\in A$ and $f\in R$;
\item
for all $b,b'\in A$ and $f\in R$ we have,
\begin{eqnarray*}
b\lact (b'\cdot f) = b'\cdot (b\lact f), \qquad (f\cdot b)\ract b'=(f\ract b')\cdot b;\\
b\lact (f\cdot b') = (b\lact f)\cdot b', \qquad (b \cdot f) \ract b' = b\cdot (f\ract b');
\end{eqnarray*}
hence $R$ is an $A\ot A^{op}$-bimodule.
\end{enumerate}
\end{lemma}

\begin{proof}
$\ul{(i)}.$ We only check the non-degeneracy of the left action. Take $f\in R$ and suppose that $b\cdot f=0$ for all $b\in A$, then $(b\cdot f)(a)=f(ab)=0$ for all $a,b\in A$. Since $A$ is idempotent, we then find that $f=0$.\\
$\ul{(ii)}.$ Again, we only proof that $R$ is non-degenerate as left $A$-module. Take $f\in R$ and suppose that 
$b\lact f=0$, for all $b\in A$. Then, for any $a\in A$, we have $0=\iota_A(b)f(a)\equal{\equref{AbimMA}}\iota_A(b\ractd f(a))$, so by the injectivity of $\iota_A$ we find $0=b\ractd f(a)=\rho_a(b)$, for all $a,b\in A$ (where we used the notation $f(a)=(\lambda_a,\rho_a)\in \MM(A)$). Hence $\rho_a=0$ for all $a\in A$, which implies that $f(a)=0$, for all $a\in A$, that is $f=0$.\\
$\ul{(iii)}.$ Easy to check.
\end{proof}

Suppose now that $A$ is a multiplier bialgebra. Suppose furthermore that the comultiplication of $A$ satisfies conditions \equref{fons1} and \equref{fons2}. Using Sweedler-type notation as introduced in the previous Section, we 
can endow $R$ with two ``local multiplication structures'', i.e. two linear maps as follows:
\begin{eqnarray*}
\ast^1: R\ot R\to \Hom(A,R),\ f\ot g \mapsto (b\mapsto f\ast^b g),\\
\ast_2: R\ot R\to \Hom(A,R),\ f\ot g \mapsto (b\mapsto f\ast_b g),
\end{eqnarray*}
where 
$f\ast^b g(a)= \mu_{\MM(A)}\big((f\ot g) (\widetilde\Delta(a)(1\ot b))\big)=  f(a_{(1,b)})g(a_{(2,b)})$ and
$f\ast_b g(a)= \mu_{\MM(A)}\big( (f\ot g) ((b\ot 1)\widetilde\Delta(a))\big)=  f(a_{(b,1)})g(a_{(b,2)})$.
These multiplications satisfy the following associativity condition.
\begin{lemma}
With notation as above, the following holds for all $f,g,h\in R$ and $a,b\in A$:
$$f\ast_a (g\ast^b h)=(f\ast_a g)\ast^b h.$$
\end{lemma}

\begin{proof}
Take any $c\in A$, then we find
\begin{eqnarray*}
\big(f\ast_a (g\ast^b h)\big)(c)&=&f(c_{(a,1)})(g\ast^b h)(c_{(a,2)})
=f(c_{(a,1)})g(c_{(a,2)(1,b)})h(c_{(a,2)(2,b)})\\
&&\hspace{-2cm}=f(c_{(1,b)(a,1)})g(c_{(1,b)(a,2)})h(c_{(2,b)})
=(f\ast_a g)(c_{(1,b)})h(c_{(2,b)}) = \big((f\ast_a g) \ast^b h\big)(c),
\end{eqnarray*}
where we used \equref{coassVD} in the third equality.
\end{proof}

Furthermore, we have the following linear maps
\begin{eqnarray*}
&\alpha\in \Hom(A,R),\ \alpha(b)(a)=\alpha_b(a)= \widetilde\varepsilon(a)\iota_A(b);\\
&\beta^1: R \to \Hom(A,R), \beta^1(f)(b)=\beta^b(f)= b\cdot f;\\
&\beta_2: R \to \Hom(A,R), \beta_2(f)(b)=\beta_b(f)= f\cdot b.
\end{eqnarray*}
Then the following unitality condition holds for the multiplicative structure we defined on $R$.
\begin{lemma}
With notation as above, we have for all $f\in R$ and $b\in A$,
$$\alpha_b \ast^1 f = \beta^1 (b\lact f),\qquad f\ast_2 \alpha_b=\beta_2(f\ract b).$$
\end{lemma}

\begin{proof}
We only show the first equality, the second one follows by a similar computation. 
Take any $a,c\in A$, then we check that
\begin{eqnarray*}
(\alpha_b \ast^c f)(a)&=& \alpha_b (a_{(1,c)}) f(a_{(2,c)}) =\iota_A(b)\widetilde\varepsilon(a_{(1,c)}) f(a_{(2,c)})\\
 &=&  \iota_A(b)f(ac) = (c\cdot (b\lact f))(a) = \beta^c(b\lact f)(a).
\end{eqnarray*}
Here we used \equref{counitVD2} in the third equation.
\end{proof}

It now makes sense to define what is a convolution inverse in $R=\Hom(A,\MM(A))$. Given $f\in R$, we say that $\bar{f}\in R$ is a \emph{(left-right) convolution inverse} of $f$ in $R$ if  
$$\bar{f} \ast^1 f = \alpha = f \ast_2 \bar{f}.$$

\begin{proposition}
Let $A$ be a multiplier Hopf algebra. Then a linear map $S\in\Hom(A,\MM(A))$ is the antipode of $A$ if and only if $S$ is a (left-right) convolution inverse of $\iota_A$ in $\Hom(A,\MM(A))$.
\end{proposition}

\begin{proof}
A linear map $S:A\to \MM(A)$ is a (left-right) convolution inverse of $\iota_A$ if and only if, for all $a,b\in A$, the following holds:
\begin{eqnarray*}
(S \ast^b \iota_A)(a) &=& S(a_{(1,b)}) \iota_A(a_{(2,b)}) \equal{\equref{AbimMA}} \iota_A(S(a_{(1,b)})\lactd a_{(2,b)})
\end{eqnarray*}
equals $\alpha_b(a) = \iota_A(b)\widetilde\varepsilon(a)$
and
\begin{eqnarray*}
(\iota_A\ast_a S)(b) &=& \iota_A(b_{(a,1)})S(b_{(a,2)}) \equal{\equref{AbimMA}} \iota_A(b_{(a,1)}\ractd S(b_{(a,2)}))
\end{eqnarray*}
is equal to 
$\alpha_a(b) =\iota_A(a)\widetilde\varepsilon(b)$.
Since $\iota_A$ is injective, these formulas are equivalent to \equref{antipode2} and \equref{antipode1}, respectively. Because of the uniqueness of the antipode of a multiplier Hopf algebra, the statement follows.
\end{proof}
 
\subsection{(Co)module algebras over a multiplier bialgebra}

By definition a multiplier bialgebra $A$ is a coalgebra in the monoidal category of non-degenerate idempotent $k$-projective algebras $\ndi_k$.

We define a \emph{right $A$-comodule algebra} as to be a right comodule over the coalgebra $A$ in $\ndi_k$. Thus a right $A$-comodule algebra $B=(B,\rho)$ consists of a non-degenerate idempotent $k$-projective algebra $B$ and a morphism $\rho:B\mor B\ot A$ such that
\begin{eqnarray}\eqlabel{coaction1}
(\rho \ot A)\circ \rho = (B\ot \Delta)\circ \rho,
\end{eqnarray} 
as morphisms from $B$ to $B\ot A \ot A$, and
\begin{eqnarray}\eqlabel{coaction2}
(B \ot \varepsilon)\circ \rho = B,
\end{eqnarray} 
as morphisms from $B$ to $B$.
Making use of the associated algebra map $\widetilde{\rho}:B\to \MM(B\ot A)$, \equref{coaction1} and \equref{coaction2} can be translated into
\begin{eqnarray}\eqlabel{coaction11}
\overline{\rho \ot A}\circ \widetilde{\rho} = \overline{B\ot \Delta}\circ \widetilde{\rho},
\end{eqnarray} 
\begin{eqnarray}\eqlabel{coaction22}
\overline{B \ot \varepsilon}\circ \widetilde{\rho} = \iota_B.
\end{eqnarray} 

\begin{proposition}
A morphism $\rho:B\mor B\ot A$ in $\ndi_k$ is coassociative in the sense of \equref{coaction11} if and only if 
\begin{eqnarray}\eqlabel{coactionfonzie1}
\overline{\rho \ot A}(\widetilde{\rho}(b)(1\ot a))&=& \overline{B \ot \Delta}(\widetilde{\rho}(b))(1\ot 1\ot a),
\end{eqnarray}
for all $a\in A$ and $b\in B$, if and only if
\begin{eqnarray}
(c\ot 1\ot 1)\overline{\rho \ot A}(\widetilde{\rho}(b)(1\ot a))&=& \overline{B \ot \Delta}((c\ot 1)\widetilde{\rho}(b))(1\ot 1\ot a),
\end{eqnarray}
for all $a\in A$ and $b,c\in B$.
\end{proposition}

\begin{proof}
Completely analogous to the proof of \prref{coass}.
\end{proof}

\begin{proposition}
A morphism $\rho:B\mor B\ot A$ in $\ndi_k$ is counital in the sense of \equref{coaction22} if and only if
\begin{eqnarray}\eqlabel{counitalcoaction}
\overline{B\ot \varepsilon}(\widetilde{\rho}(b)(1\ot a))&=& \widetilde{\varepsilon}(a)\iota_B(b),
\end{eqnarray}
for all $a\in A$ and $b\in B$.
\end{proposition}

\begin{proof}
The proof is quite analogous to the one of \prref{counit}. We will proof one direction. 
Suppose that \equref{coaction22} holds. For $b'=\sum \widetilde{\varepsilon}(g)b'^1b'^2\in B=B^2$ we have that
\begin{eqnarray*}
&&\hspace{-2cm}
\overline{B\ot \varepsilon}(\widetilde{\rho}(b)(1\ot a))\lactd b'
= \sum \widetilde{B\ot \varepsilon} \big((\widetilde{\rho}(b)(1\ot a))\lactd (b'^1\ot g)\big)\lactd b'^2\\
&=& \sum \widetilde{B\ot \varepsilon} (\widetilde{\rho}(b)\lactd (b'^1\ot ag))\lactd b'^2= \sum \overline{B\ot \varepsilon}(\widetilde{\rho}(b))\lactd \widetilde{\varepsilon}(a)b'\\
&\equal{\equref{coaction22}}& \iota_B(b)\lactd \widetilde{\varepsilon}(a)b'= \widetilde{\varepsilon}(a)\iota_B(b)\lactd b',
\end{eqnarray*}
hence the left multipliers of both sides of \equref{counitalcoaction} are equal for all $a\in A$ and $b\in B$; the equality of the right multipliers follows in a similar way. Here we used in the third equality that $\widetilde{\varepsilon}(a)b'=\sum \widetilde{\varepsilon}(a)\widetilde{\varepsilon}(g)b'^1b'^2=\sum \widetilde{\varepsilon}(ag)b'^1b'^2$.
\end{proof}

By a \emph{right $A$-module algebra} we mean a non-unital algebra in the monoidal category $\Mm_A$ of non-degenerate idempotent $k$-projective right $A$-modules. That is an algebra $B$ that is at the same time a non-degenerate idempotent $k$-projective right $A$-module, such that the multiplication $\mu_B: B\ot B\to B$ is a right $A$-linear map. The latter means   that 
$$\mu_B(b\ot b')\cdot a= \mu_B((b\ot b')\cdot a) =\mu_B((b\ot b')\ractd \widetilde\Delta(a)),$$
for all $b,b'\in B$ and $a\in A$,
and where we used the right $A$-action on $B\ot B$ as in \leref{compatibleAmod}.

\begin{remark}\relabel{remcoaction}
In \cite{VDZha:Galois} comodule algebras and module algebras were defined over a multiplier Hopf algebra. Just like multiplier Hopf algebras are in particular multiplier bialgebras, (co)module algebras over a multiplier Hopf algebra are particular instances of (co)module algebras over the underlying multiplier bialgebra. Let us make this correspondence a bit more explicit.
Note that \equref{coactionfonzie1} is an equality in $\MM(B\ot A \ot A)$.  
In \cite{VDZha:Galois}, a so-called coassociative coaction of $A$ on $B$ is an algebra map $\widetilde{\rho}:B\to \MM(B\ot A)$ such that
a similar coassociativity condition as \equref{coactionfonzie1} holds, but under the extra assumption that
\begin{eqnarray}
\widetilde{\rho}(b)(1\ot a)&\subseteq& \iota_{B\ot A}(B\ot A),\eqlabel{ff1}\\
(1\ot a)\widetilde{\rho}(b)&\subseteq& \iota_{B\ot A}(B\ot A),\eqlabel{ff2}
\end{eqnarray}
for all $a\in A$ and $b\in B$. These equations \equref{ff1} and \equref{ff2} ``force'' the equality \equref{coactionfonzie1} to be in $\iota_{B\ot A\ot A}(B\ot A \ot A)$, hence obtaining  an equality in $B\ot A\ot A$ (by the injectivity of $\iota_{B\ot A\ot A}$). It is actually this equality that expresses the coassociativity in \cite{VDZha:Galois}. Furthermore, \equref{counitalcoaction} states that the coaction
$\widetilde{\rho}:B\to \MM(B\ot A)$ is counital in the sense of \cite{VDZha:Galois}.

\end{remark}

\subsection*{Acknowledgement}
The authors would like to thank A.\ Van Daele for explaining them the concept of multiplier Hopf algebra during the joint Algebra seminars of the universities of Antwerpen, Brussel, Leuven and Hasselt in 2008.
We thank Matthew Daws for pointing out references \cite{Dau:mult}, \cite{Ho:Cohom} and \cite{Joh:centr}.

JV thanks the Fund for Scientific Research--Flanders (Belgium) (F.W.O-Vlaanderen) for a Postdoctoral Fellowship.

\def\cprime{$'$}


\begin{thebibliography}{1}
\expandafter\ifx\csname url\endcsname\relax
  \def\url#1{{\tt #1}}\fi
\expandafter\ifx\csname urlprefix\endcsname\relax\def\urlprefix{URL }\fi
\providecommand{\eprint}[2][]{\url{#2}}

\bibitem{bohm:weakhopf}
G.~B{\"o}hm, F.~Nill and K.~Szlach{\'a}nyi, Weak {H}opf algebras. {I}.
  {I}ntegral theory and {$C\sp *$}-structure, {\em J. Algebra\/}, {\bf
  221}~(2), (1999) 385--438.

\bibitem{BohmSzl:hgdax}
G.~B{\"o}hm and K.~Szlach{\'a}nyi, Hopf algebroids with bijective antipodes:
  axioms, integrals, and duals, {\em J. Algebra\/}, {\bf 274}~(2), (2004)
  708--750.

\bibitem{CDL}
S.~Caenepeel and M.~De~Lombaerde, A categorical approach to {T}uraev's {H}opf
  group-coalgebras, {\em Comm. Algebra\/}, {\bf 34}~(7), (2006) 2631--2657.

\bibitem{Dau:mult}
J.~Dauns, Multiplier rings and primitive ideals, 
{\em Trans. Amer. Math. Soc.\/}, {\bf 145},  (1969) 125--158. 

\bibitem{Drin:QH}
V.~G. Drinfel{\cprime}d, Quasi-{H}opf algebras, {\em Algebra i Analiz\/}, {\bf
  1}~(6), (1989) 114--148.

\bibitem{Ho:Cohom}
G.~Hochschild, Cohomology and representations of associative algebras,
{\em Duke Math. J.\/}, {\bf 14},  (1947) 921--948. 

\bibitem{Joh:centr}
B.~E.~Johnson, An introduction to the theory of centralizers,
{\em Proc. London Math. Soc. (3)\/}, {\bf 14},  (1964) 299--320. 

\bibitem{McLane}
S.~Mac~Lane, Categories for the working mathematician, volume~5 of {\em
  Graduate Texts in Mathematics\/}, Springer-Verlag, New York, 1998, second
  edition.

\bibitem{Sch:biHopf}
P.~Schauenburg, A bialgebra that admits a {H}opf-{G}alois extension is a {H}opf
  algebra, {\em Proc. Amer. Math. Soc.\/}, {\bf 125}~(1), (1997) 83--85.

\bibitem{VD:Mult}
A.~Van~Daele, Multiplier {H}opf algebras, {\em Trans. Amer. Math. Soc.\/}, {\bf
  342}~(2), (1994) 917--932.

\bibitem{VDZha:corep}
A.~Van~Daele and Y.~Zhang, Corepresentation theory of multiplier {H}opf
  algebras. {I}, {\em Internat. J. Math.\/}, {\bf 10}~(4), (1999) 503--539.

\bibitem{VDZha:Galois}
A.~Van~Daele and Y.~H. Zhang, Galois theory for multiplier {H}opf algebras with
  integrals, {\em Algebr. Represent. Theory\/}, {\bf 2}~(1), (1999) 83--106.

\end{thebibliography}
\end{document}